\documentclass[12pt]{amsart}
\usepackage{xy}
\xyoption{all}
\usepackage[utf8]{inputenc}
\usepackage{algpseudocode}
\usepackage[margin=2.6cm]{geometry}
\usepackage{amssymb,mathtools,mathrsfs,bbm}
\usepackage{enumerate}
\usepackage{graphicx,titletoc,xcolor}
\definecolor{darkblue}{rgb}{0.0, 0.0, 0.55}
\definecolor{bordeaux}{rgb}{0.34, 0.01, 0.1}
\usepackage[colorlinks,linkcolor=bordeaux,citecolor=darkblue,urlcolor=black,hypertexnames=true]{hyperref}
\usepackage{tikz-cd}

\newcommand{\dashto}{\to\mkern-21mu {\color{white}ll}\mkern9.6mu}

\makeatletter
\def\moverlay{\mathpalette\mov@rlay}
\def\mov@rlay#1#2{\leavevmode\vtop{
		\baselineskip\z@skip \lineskiplimit-\maxdimen
		\ialign{\hfil$#1##$\hfil\cr#2\crcr}}}
\makeatother
\newcommand{\lfre}{{<}}
\newcommand{\rfre}{{>}}
\newcommand{\lfree}{\moverlay{(\cr<}}
\newcommand{\rfree}{\moverlay{)\cr>}}
\newcommand{\kk}{{\mathbbm k}}
\newcommand{\px}[1]{\kk\lfre #1\rfre}
\newcommand{\rx}[1]{\kk\lfree #1\rfree}

\DeclareMathOperator{\GL}{GL}

\DeclareMathOperator{\id}{id}

\DeclareMathOperator{\tr}{tr}
\DeclareMathOperator{\kar}{char}

\newcommand{\N}{{\mathbb N}}

\newcommand{\Z}{{\mathbb Z}}
\newcommand{\cB}{{\mathcal B}}
\newcommand{\cR}{{\mathcal R}}
\newcommand{\cL}{{\mathcal L}}
\newcommand{\cT}{{\mathcal T}}

\newcommand{\cC}{{\mathscr C}}
\newcommand{\cD}{{\mathscr D}}
\newcommand{\cZ}{{\mathscr Z}}

\newcommand{\sD}{{\mathbf D}}
\newcommand{\sW}{{\mathbf W}}

\newcommand{\fz}{{\mathfrak z}}
\newcommand{\fX}{{\mathfrak X}}
\newcommand{\fZ}{{\mathfrak Z}}
\newcommand{\op}{{\rm op}}

\renewcommand{\d}{{\rm d}}

\newtheorem{theorem}{Theorem}[section]
\newtheorem{proposition}[theorem]{Proposition}
\newtheorem{lemma}[theorem]{Lemma}
\newtheorem{corollary}[theorem]{Corollary}
\newtheorem{question}[theorem]{Question}

\makeatletter
\newtheorem*{rep@thm}{\rep@title}
\newcommand{\newreptheorem}[2]{%
	\newenvironment{rep#1}[1]{%
		\def\rep@title{#2 \ref{##1}}%
		\begin{rep@thm}}%
		{\end{rep@thm}}}
\makeatother
\newreptheorem{thm}{Theorem}
\newreptheorem{cor}{Corollary}

\theoremstyle{definition}

\newtheorem{algorithm}[theorem]{Algorithm}
\newtheorem{example}[theorem]{Example}

\newcommand{\Hom}{\operatorname{Hom}}

\newcommand{\End}{\operatorname{End}}

\newcommand{\divisor}{\operatorname{div}}
\title[Invariants of finite groups acting on (free) skew fields]{Invariants of finite groups\\acting on (free) skew fields}

\author[Harm Derksen]{Harm Derksen${}^\ast$}
\address{Department of Mathematics, Northeastern University, Massachusetts, USA}
\email{ha.derksen@northeastern.edu}
\thanks{${}^\ast$Partially supported by NSF grant DMS-2147769 and a Simons Foundation Fellowship.}

\author[Jurij Vol\v{c}i\v{c}]{Jurij Vol\v{c}i\v{c}${}^\dagger$}
\address{Department of Mathematics, University of Auckland, New Zealand}
\email{jurij.volcic@auckland.ac.nz}
\thanks{${}^\dagger$Partially supported by NSF grant DMS-2348720.}

\date{\today}

\begin{document}

\begin{abstract}
Let $M$ be a finitely generated skew field over a ground field $\kk$, and let $G$ be a finite group of $\kk$-linear automorphisms of $M$. 
This paper investigates finite generation of the skew subfield $M^G$ of $G$-invariants in $M$, and relations between the generators. 
The first main result shows that $M^G$ is finitely generated. Stronger conclusions hold when $M$ is a free skew field, i.e., the universal skew field of fractions of a free algebra. 
The second main result is the solution of the free Noether problem for non-modular linear group actions: if $G$ acts linearly on the free skew field $M$ on $m$ generators and $\kar\kk$ does not divide $|G|$, then $M^G$ is the free skew field on $|G|(m-1)+1$ generators. 
In contrast, a nonlinear action of $\Z_2$ on the free skew field $M$ on two generators is presented such that $M^{\Z_2}$ is not a free skew field, resolving the free L\"uroth problem. This action also exposes a non-scalar element of $M$ whose centralizer is not a rational field, refuting a conjecture of P.~M.~Cohn from 1978.
\end{abstract}

\maketitle

\tableofcontents

\section{Introduction}\label{sec:intro}

Invariant theory studies group actions on algebraic varieties in terms of functions on them \cite{PV,Olv,DK}. Its classical problems focus on identifying generators and relations for polynomial and rational invariants of a group acting on an affine space. 
In particular, given a finite group acting on a field of multivariate rational functions, Noether's problem asks whether the invariant subfield is itself rational \cite{Noe,Swa,Swa1}. Emmy Noether originally considered permutation actions, but the problem has also been studied for linear and non-linear actions of finite groups.
Negative answers to this question in particular resolved L\"uroth's problem on existence of non-rational subfields in a rational field \cite{CG,AM,Salt}. 
These developments inspired analogous problems in noncommutative rings, such as the Noether problem for rings of differential operators \cite{AD,FS}.

This paper investigates rational invariants in freely noncommutative variables (i.e., not satisfying any relations). Let $\kk$ be a ground field, and let $G$ be a finite group of automorphisms acting on the free algebra of noncommutative polynomials $\px{x_1,\dots,x_m}$. The first structural result about the subalgebra of $G$-invariants $\px{x_1,\dots,x_m}^G$ when $G$ is the full symmetric group permuting the variables $x_j$ dates back to \cite{Wol}. In general, when $G$ acts linearly on $x_1,\dots,x_m$, 
the algebra $\px{x_1,\dots,x_m}^G$ is free on infinitely many generators, unless $G$ simply scales $x_1,\dots,x_m$ by the same root of unity, in which case it is free and finitely generated \cite{DF,Kha}. Although this result conclusively describes the ring structure of noncommutative polynomial invariants, several further advances were made on their combinatorial \cite{Gel,RS} and analytic aspects \cite{AY,Cus}.

In contrast, noncommutative rational invariants have been less studied. The free algebra $\px{x_1,\dots,x_m}$ admits the universal skew field of fractions $\rx{x_1,\dots,x_m}$, also called the \emph{free skew field} on $m$ generators \cite{Ami}. 
Beyond its universal role among division rings \cite{Coh1,Sch}, the free skew field appears in automata theory \cite{BR1}, control systems \cite{BGM} and computational complexity \cite{HW,Gar}. 
Furthermore, its elements can be viewed as noncommutative rational functions in several matrix variables \cite{KVV1}; this analytic perspective connects the free skew field to free analysis \cite{KVV2,AMY} and free probability \cite{HMS,CMMPY}. 

If a finite group $G$ acts on $\px{x_1,\dots,x_m}$, its action naturally extends to $\rx{x_1,\dots,x_m}$. Although $\px{x_1,\dots,x_m}^G$ is almost never finitely generated as an algebra as explained above, it was observed in \cite{AMY,Cus} that when $G$ is abelian and acts linearly, there are \emph{finitely} many $G$-invariant noncommutative rational functions such that every element of $\px{x_1,\dots,x_m}^G$ is a \emph{rational} function in them. 
In the first systematic study of the structure of noncommutative rational invariants, \cite{KPPV} considered linear actions of finite solvable groups $G$ on $\rx{x_1,\dots,x_m}$. It was shown that $\rx{x_1,\dots,x_m}^G$ is a finitely generated skew field in this case, and is free on $|G|(m-1)+1$ generators for certain linear actions of $G$ satisfying a cohomological condition \cite{Pod1} (which is always satisfied for finite abelian groups).

This paper explores finite generation and freeness of $\rx{x_1,\dots,x_m}^G$ for a general action of a finite group $G$. Our methods are independent of the preceding results mentioned above, and arise from finite extensions of skew fields. In fact, the scope of the first main result (Theorem~\ref{theo:FinitelyGenerated} below) on finite generation extends well beyond free skew fields.
If $L\subseteq M$ are skew fields, then $M$ is both a  free left $L$-module and a free right $L$-module. The left $L$-dimension (resp.\ right $L$-dimension) of $M$ is the number of free generators of $M$ as a left (resp.\ right) $L$-module.  The left and right dimension do not have to be the same, and one could even be infinite while the other is finite (see~\cite[Section 5.9]{Coh1}, \cite[Chapter 5]{Coh3} or~\cite{Sch0}).

\begin{repthm}{theo:FinitelyGenerated} 
Let $L\subseteq M$ be skew fields containing a field $\kk$. If $M$ is generated by $m$ elements over $\kk$ 
and the left or right $L$-dimension of $M$ is $d$, 
then $L$ is generated by at most $d^2(m-1)+d$ elements over $\kk$.
\end{repthm}

If $L=M^G$ for a finite group $|G|$, then $M$ has (left or right) dimension at most $|G|$ over $L$, so Theorem \ref{theo:FinitelyGenerated} 
applies to invariant skew subfields. 
When $M$ is a finitely generated free skew field, we obtain the generators of $L$ in a constructive manner, as well as a finite set of relations between them that are fundamental in a certain sense (Proposition \ref{p:freetest}). 
Furthermore, these relations can be resolved if the action of $G$ is linear; that is, we obtain an affirmative answer to the free Noether problem in this case.

\begin{repthm}{t:linfree}
If $G$ is a finite group of linear automorphisms on $\rx{x_1,\dots,x_m}$ and $\kar\kk$ does not divide $|G|$, then $\rx{x_1,\dots,x_m}^G$ is the free skew field on $|G|(m-1)+1$ generators.
\end{repthm}

At this point, it is worth mentioning that to date it has been unknown whether every skew subfield of a free skew field is free \cite[pages 183 and 206]{Sch}. We show that this free L\"uroth problem has a negative resolution by demonstrating that the conclusion of Theorem~\ref{t:linfree} fails for nonlinear group actions. More concretely, Theorem \ref{t:nonfree} presents an action of $\Z_2$ on $\rx{x,y}$ such that $\rx{x,y}^{\Z_2}$ is not free. 

A special case of the free L\"uroth problem is related to the rationality of centralizers in free skew fields. In \cite[Conjecture]{Coh0}, Cohn conjectured that the centralizer of a non-scalar element in a free skew field is isomorphic to $\kk(t)$; see also \cite[page 183]{Sch}, \cite[page 329]{Lot} and \cite[Problem 4.1]{Bel} for other reappearances. 
However, our example of a non-free skew subfield of $\rx{x,y}$ leads to a counter-example to this conjecture.

\begin{repcor}{c:centraliser}
Let $\kar\kk\neq2$. There exists $r\in \rx{x,y}\setminus\kk$ with a non-rational centralizer. 
\end{repcor}

\section{Preliminaries}\label{sec:prelim}

Throughout this paper let $\kk$ be a fixed ground field. All rings and skew fields are considered to be unital $\kk$-algebras, and homomorphisms between them are unital and $\kk$-linear. Given two subsets $A,B$ of an algebra, let $AB$ denote the span of all products $ab$ for $a\in A$ and $b\in B$. When discussing a group action on an algebra, we mean an action by $\kk$-linear automorphisms (i.e., an action that is compatible with the $\kk$-algebra structure), and we write it using the exponential notation. Given an algebra $R$ and a group $G$ acting on $R$, let $R^G$ denote the subalgebra of $G$-invariant elements.

Let $\px{x_1,\dots,x_m}$ be the free associative algebra generated by freely noncommuting variables $x_1,\dots,x_m$. 
We say that $\px{x_1,\dots,x_m}$ is free of rank $m$. 
Let $\rx{x_1,\dots,x_m}$ denote the universal skew field of fractions of $\px{x_1,\dots,x_m}$; see   \cite[Corollary 4.5.9 and Theorem 5.4.1]{Coh1} or \cite[Corollaries 2.5.2 and 7.5.14]{Coh2} for uniqueness and universality. Informally, $\rx{x_1,\dots,x_m}$ is a skew field where the only relations between $x_1,\dots,x_m$ are consequences of division ring axioms, such as 
\begin{equation}\label{e:example}
\big(x_1-x_2x_4^{-1}x_3\big)^{-1} = 
x_1^{-1}+x_1^{-1}x_2\big(x_4-x_3x_1^{-1}x_2\big)^{-1}x_3x_1^{-1}.
\end{equation}
Therefore, $\rx{x_1,\dots,x_m}$ is called the \emph{free skew field} of rank $m$. 
For the reader's convenience, we collect some facts about free skew fields. 

\begin{enumerate}[(a)]
\item Consider the set of formal rational expressions in $x_1,\dots,x_m$ and scalars from $\kk$ that are well-defined at some $m$-tuple of square matrices over $\kk$. Two such expressions are equivalent if their evaluations agree on all matrix tuples over $\kk$ at which both expressions are well-defined. By \cite[Theorem 16 and Section II.1]{Ami} or \cite[Section 2]{KVV1}, one can view elements of $\rx{x_1,\dots,x_m}$ as equivalence classes of such formal rational expressions.

\item By \cite{Lew}, another way to realize $\rx{x_1,\dots,x_m}$ is as the skew subfield generated by $\px{x_1,\dots,x_m}$ within the skew field of Malcev-Neumann series over the free group generated by $x_1,\dots,x_m$.

\item A square matrix $A$ over $\px{x_1,\dots,x_m}$ is full if in any factorization $A=BC$ over $\px{x_1,\dots,x_m}$, $B$ has at least as many columns as $A$. By \cite[Corollary 7.5.14]{Coh2}, a square matrix over $\px{x_1,\dots,x_m}$ is invertible over $\rx{x_1,\dots,x_m}$ if and only if it is full.

\item By \cite[Theorem 5.8.11 and Corollary 5.8.14]{Coh1}, free algebras and free skew fields are determined by their ranks up to isomorphism.
\end{enumerate}

Let $a_1,\dots,a_m$ be elements of a skew field $D$, and let $r\in\rx{x_1,\dots,x_m}$. If a formal rational expression representing $r$ is well-defined at $(a_1,\dots,a_m)$, in the sense that all nested inverses in the expression exist in $D$ when the $x_i$ are replaced by the $a_i$, then we say that $(a_1,\dots,a_m)$ is in the domain of $r$. If two expressions representing $r$ (e.g., as in \eqref{e:example}) are well-defined at $(a_1,\dots,a_m)$, their evaluations at it coincide \cite[Section 7.3]{Coh1}, and are denoted $r(a_1,\dots,a_m)\in D$. In particular, whenever referring to the value $r(a_1,\dots,a_m)$, we tacitly assume that $(a_1,\dots,a_m)$ is in the domain of $r$. More generally, if $A$ is a matrix over $\rx{x_1,\dots,x_m}$ such that $(a_1,\dots,a_m)$ lies in the domains of all its entries, then $A(a_1,\dots,a_m)$ is a well-defined matrix over $D$. 

Since (unital) homomorphisms between skew fields are necessarily embeddings, we also use more general maps between them. 
Let $D$ and $E$ be skew fields. A \emph{local homomorphism} from $D$ to $E$ \cite[Section 4.1]{Coh1} is a unital homomorphism from a subring $D_0$ of $D$ to $E$ that maps non-invertible elements in $D_0$ to zero.  

\begin{lemma}[{\cite[Theorem 4.4.1]{Coh1}}]\label{l:univ}
Let $D$ be a skew field generated by $a_1,\dots,a_m$, and let $E$ be a skew field containing $b_1,\dots,b_m$. The following are equivalent:
\begin{enumerate}
    \item there exists a local homomorphism from $D$ to $E$ such that $a_i\mapsto b_i$ for $i=1,\dots,m$;
    \item for every square matrix $A$ over $\px{x_1,\dots,x_m}$, if $A(b_1,\dots,b_m)$ is invertible over $E$ then $A(a_1,\dots,a_m)$ is invertible over $D$.
\end{enumerate}
\end{lemma}

The universal property of $\rx{x_1,\dots,x_m}$ then states that for every skew field $E$ with elements $b_1,\dots,b_m$, there is a local homomorphism from $\rx{x_1,\dots,x_m}$ to $E$ mapping $x_i$ to $b_j$; equivalently, invertibility of $A(b_1,\dots,b_m)$ over $E$ (where $A$ is a matrix over $\px{x_1,\dots,x_m}$) implies invertibility of $A(x_1,\dots,x_m)$ over $\rx{x_1,\dots,x_m}$. This universal property is a consequence of the property of full matrices in item (c) above.

\section{Finite skew field extensions}\label{sec:general}

In this section we establish finite generation of skew fields whose finite extensions are finitely generated (Theorem \ref{theo:FinitelyGenerated}). For a comprehensive study of skew field extensions and the corresponding Galois theory, see \cite[Section 3]{Coh1} and \cite{Dre}.

Let $L\subseteq M$ be skew fields. To emphasize the left $L$-module and right $L$-module structure of $M$, we write ${}_LM$ respectively $M_L$. Then $\End(M)$, $\End({}_LM)$ and $\End(M_L)$ are the algebras of $\kk$-vector space endomorphisms,
left $L$-module endomorphisms and right $L$-module endomorphisms respectively. Similarly, $\dim M$, $\dim{}_LM$, $\dim M_L$ are the dimensions of $M$ as a $\kk$-vector space, left $L$-module and right $L$-module respectively. In general, $\dim M_L\neq \dim{}_LM$ (see \cite[Section 5.9]{Coh1} for an example). We start with a technical lemma.

\begin{lemma}\label{lem:complement}
Let $A_1,\dots,A_r,B$ be modules over a ring, and assume $A_1,\dots,A_r$ are simple. Then, there exists a subset $I\subseteq \{1,\dots,r\}$ such that 
$B+\sum_{i=1}^r A_i=B\oplus \bigoplus_{i\in I}A_i$.
\end{lemma}

\begin{proof}
We prove the statement by induction on $i$. The case $i=0$ is obvious. If $A_r\cap B=\{0\}$
then we have $B+\sum_{i=1}^r A_i=A_r\oplus B+\sum_{i=1}^{r-1}A_i$ and we use induction. Otherwise, $A_r\cap B=A_r$ because $A_r$ is simple. In this case,  $A_r\subseteq B$ and we have $B+\sum_{i=1}^r A_i=B+\sum_{i=1}^{r-1} A_i$. Again, the lemma follows by induction.
\end{proof}

Suppose that $L$ and $M$ are skew fields with $L\subseteq M$, and $V$ is an $m$-dimensional $\kk$-subspace of $M$ that generates $M$ as a skew field over $\kk$. If we set $V_+=\kk+V$, then $\kk\lfre V\rfre=\bigcup_{i=0}^\infty V_+^i$ is the associative algebra generated by $V$. In the case where $\dim M_L$ is finite, the proposition  below constructs a $\kk$-subspace $W\subseteq \kk\lfre V\rfre$ such that $WL=M$ and $\dim W=\dim M_L$. This means that any $\kk$-basis of $W$ is a basis of $M$ as a right $L$-module. The subspace $W$ is constructed inductively and greedily as a direct sum $W=W_1\oplus W_2\oplus \cdots$ such that $W\cap V_+^k=W_1\oplus W_2\oplus \cdots \oplus W_k$ contains an $L$-basis of $V_+^kL$ for all $k$. 
This proposition will be useful for finding small generating sets.

\begin{proposition}\label{prop:Ws}
Let $L\subseteq M$ be skew fields. Suppose an $m$-dimensional $\kk$-subspace $V\subseteq M$ generates the skew field $M$ over $\kk$, and $\dim M_L$ is finite.
There exists a $\kk$-subspace $W$ of $M$ with the following properties:
\begin{enumerate}
    \item $WL=M$ and $\dim W=\dim M_L$;
    \item $\kk\subseteq W\subseteq \kk+VW$.
\end{enumerate}
\end{proposition}

\begin{proof}
Denote $V_+=V+\kk\subseteq M$. Below we use the conventions $V^0=V_+^0=\kk$.
We  inductively construct $\kk$-subspaces $W_i\subseteq M$ for $i\geq 0$ such that $W_0=\kk$ and  the following properties hold:
\begin{enumerate}
   \item $\dim W_i=\dim (W_iL)_L$ for $i\geq 0$;
   \item $V_+^iL+VW_iL=V_+^{i+1}L$ for $i\geq 0$;
    \item $W_i\subseteq VW_{i-1}\subseteq V^i$ for $i\geq 1$;
    \item  $V_+^{i-1}L\oplus W_iL=V_+^iL$ for $i\geq 1$.
\end{enumerate}
It is clear that (1) and (2) hold for $i=0$. Suppose that $i\geq 0$ and  we have already constructed the spaces $W_0,W_1,\dots,W_{i}$. We will construct $W_{i+1}$ with properties (1)--(4).

Choose a basis $a_1,\dots,a_r$ in $VW_{i}$ and let $A_i=a_iL$.
Then we have $V_+^{i}L+\sum_{j=1}^r A_j=V_+^{i+1}L$.
By Lemma \ref{lem:complement}, 
there exists a subset $I\subseteq \{1,\dots,r\}$ 
such that $V_+^{i}L\oplus \bigoplus_{j\in I} A_j=V_+^{i+1}L$.
Let $W_{i+1}$ be the $\kk$-span of $a_j$ for $j\in I$. 
Then, we have 
$V_+^{i}L\oplus W_{i+1}L=V_+^{i+1}L$ and $\dim W_{i+1}=|I|=\dim(W_{i+1}L)_L$. 
Thus, $W_{i+1}$ has properties (1), (3) and (4).
By (3), we have $W_{i+1}\subseteq V^{i+1}$, so 
\begin{multline*}
V_+^{i+1}L+VW_{i+1}L=(V_+^{i+1}+W_{i+1})L+VW_{i+1}L=
V_+^{i+1}L+(\kk+V)W_{i+1}L=
\\
=
V_+^{i+1}L+V_+W_{i+1}L=V_+(V_+^{i}L\oplus W_{i+1}L)=V_+V_+^{i+1}L=V_+^{i+2}L.
\end{multline*}

By induction it is clear that
$V_+^jL=\bigoplus_{i=0}^j W_iL$ and 
$$
\dim M_L \geq \dim (V_+^jL)_L=\sum_{i=0}^j \dim (W_iL)_L=\sum_{i=0}^j \dim W_i=\dim \bigoplus_{i=0}^j W_i.$$
 This shows that for some $s\in\N$, we have $0=W_{s+1}=W_{s+2}=\cdots$.
Define $W=\bigoplus_{i=0}^s W_i$. Consider the set 
$N=\{x\in M\colon xWL\subseteq WL\}$. It is clear that 
$N$ is a subalgebra of $M$. If $x\in N$ is nonzero,
then left multiplication by $x$ gives an isomorphism between $xWL$ and $WL$ as right $L$-modules. So $(xWL)_L$ is a right $L$-submodule of $(WL)_L$ of the same dimension. It follows that $xWL=WL$. But then $x^{-1}WL\subseteq WL$, 
and so $x^{-1}\in N$. 
This shows that $N$ is a skew subfield of $M$. It also contains
the space $V$ that generates $M$, so $N=M$.
Now  we have $(WL)_L=M_L$ and $\dim W=
\sum_{i=0}^s \dim W_i=\sum_{i=0}^s \dim (W_iL)_L=\dim (WL)_L=\dim M_L$.
Finally, we have $\kk=W_0\subseteq W=\kk+(\sum_{i=1}^s W_i)\subseteq \kk+\sum_{i=1}^s VW_{i-1}\subseteq \kk+VW$.
\end{proof}

\subsection{General finite generation}

Our first main result shows that skew fields with finitely generated finite extensions are themselves finitely generated.

\begin{theorem}\label{theo:FinitelyGenerated}
Let $L\subseteq M$ be skew fields. Suppose $M$ is generated by $m$ elements over $\kk$, and $d=\dim M_L$ is finite. Then, $L$ is generated over $\kk$ by at most $d^2(m-1)+d$ elements.
\end{theorem}

\begin{proof}
    Let $V$ be the $\kk$-linear span of the $m$ generators of $M$ over $\kk$. We use Proposition \ref{prop:Ws} to construct a
    subspace $W$ of $M$ with $WL=M$, $\dim W=\dim M_L=d$ and $\kk\subseteq W\subseteq \kk+VW$. 
    As right $L$-modules we have
    $M_L=(WL)_L\cong (W\otimes L)_L$.
    This gives an isomorphism $\End(M_L)\cong \End((W\otimes L)_L)=\End(W)\otimes L$. For $x\in M$, let $\lambda_x\in \End(M_L)$ be left multiplication by $x$. 
    Define $\rho:M\to \End(W)\otimes L$ by $\rho(x)=\lambda_x$.
    It is clear that $\rho$ is a $\kk$-algebra homomorphism and $\rho$ is injective because $M$ is a skew field and $\rho(1)\neq 0$.

We have isomorphisms
$$\Hom(M,\End(W)\otimes L)\cong \Hom(M\otimes W,W\otimes L)\cong \Hom(M\otimes W\otimes W^\star,L).$$ 
Thus, we can view 
$\rho$ in 3 different ways, namely
\begin{eqnarray*}
    \rho:M &\to & \End(W)\otimes L,\\
    \rho':M\otimes W &\to&  W\otimes L\\
    \rho'':M\otimes W\otimes W^\star & \to &  L.
\end{eqnarray*} 
    Since $W\subseteq \kk+VW$, the subspace $W\cap VW$ has dimension at least $d-1$. Note that $(VW\cap W)\otimes \kk$ is contained in $\rho'(V\otimes W)$. 
    So there exists a subspace $T\subseteq V\otimes W$ with $\dim T=\dim (W\cap VW)$ for which $\rho'(T)= (W\cap VW)\otimes \kk$. This means that $\rho''(T\otimes W^\star)= \kk$. 
    Let $T^\perp$ be a complement of $T$ in $V\otimes W$; its dimension is at most $dm-(d-1)$. Let $Z=\rho''(T^\perp\otimes W^\star)$. Then, 
    $$\dim Z\leq \dim T^{\perp}\otimes W^\star\leq  (dm-d+1)d=d^2(m-1)+d.$$
We have $\rho'(V\otimes W)=\rho'(T\oplus T^\perp)=\rho'(T)+\rho'(T^\perp)\subseteq W\otimes \kk+W\otimes Z$ and
$\rho(V)\subseteq \End(W)\otimes (\kk+Z)$.
Since $\rho(V)$ is contained in the algebra 
   $\End(W)\otimes \kk(Z)$ and $V$ generates $M$, we get $\rho(M)\subseteq \End(W)\otimes \kk(Z)$.
   If $x\in L$, then we have $x=\lambda_x(1)=\rho(x)(1\otimes 1)\in \kk(Z)$ because $\rho(x)\in \End(W)\otimes \kk(Z)$.
   Since $L\subseteq \kk(Z)$ and $Z$ is contained in $L$, we have $L=\kk(Z)$. This shows that $L$ is generated by at most $d^2(m-1)+d$ elements. 
\end{proof}

\begin{corollary}
Suppose $M$ is generated by $m$ elements over $\kk$, and $d=\dim {}_LM$ is finite. 
Then, $L$ is generated over $\kk$ by at most $d^2(m-1)+d$ elements.
\end{corollary}

\begin{proof}
Taking the opposite rings, we have $\kk\subseteq L^{\rm op}\subseteq M^{\rm op}$ and $\dim M^{\rm op}_{L^{\rm op}}=\dim{}_LM=d<\infty$. By the previous theorem, $M^{\rm op}$ is generated by at most $d^2(m-1)+d$ generators, so the same is true for $M$.    
\end{proof}

A natural source of finite skew field extensions are invariants of finite groups. 
To see this, we first record a weak noncommutative version of Dedekind's lemma (for a different generalization, see \cite[Theorem 4]{Dre} or \cite[Theorem 3.3.2]{Coh1}).

\begin{lemma}\label{l:dedekind}
Let $\varphi_1,\dots,\varphi_d$ be distinct automorphisms of a skew field $M$, and let $L$ be their fixed skew subfield in $M$. Then, $\varphi_1,\dots,\varphi_d$ are $L$-linearly independent in $\End(M_L)$.
\end{lemma}

\begin{proof}
Let $W=\operatorname{span}_L\{\varphi_1,\dots,\varphi_d\}\subseteq\End(M_L)$. It suffices to find $x_1,\dots,x_d\in M$ such that $W_k=\{\varphi\in W\colon \varphi(x_1)=\cdots=\varphi(x_k)=0\}$ for $k=1,\dots,d$ has dimension $d-k$ as a right $L$-module ($k=d$ then implies the desired statement). 
We find the $x_j$ by induction on $k$. For $k=1$ take any nonzero $x_1\in M$. Now suppose that $\dim (W_{k-1})_L=d-(k-1)$. Choose a nonzero $\psi\in W_{k-1}$; its kernel is a proper $L$-submodule of $M$, so there exists $x_k\in M\setminus\ker\psi$. Then, $W_k=\{\varphi\in W_{k-1}\colon \varphi(x_k)=0\}$ has dimension $d-k$ over $L$.
\end{proof}

For a finite group $G$, we denote the group algebra of $G$ over a skew field $L$ by $L[G]$.
Suppose a finite group $G$ acts on a skew field $M$ by automorphisms. 
If we apply an automorphism $g\in G$
 to an element $y\in M$ then we may write $y^g$ for the result instead of $g(y)$ to ease the notation.
 The set $M^G$ of all elements of $M$ that are fixed by the group 
 is again a skew field. 
 
For later refinements, we require the following noncommutative analogs of Artin's theorem and the normal basis theorem. 
We say that a group $G$ of automorphisms on a skew field $M$ consists of outer automorphisms if the identity is the only inner automorphism of $M$ in $G$.

\begin{lemma}\label{l:fixed}
Let $G$ be a finite group of automorphisms on a skew field $M$, and $L=M^G$.
\begin{enumerate}
\item $\dim M_L,\dim {}_L M\le |G|$.
\item If $G$ is a group of outer automorphisms, then $M\cong\kk[G]\otimes L$ as an $L$-module and a $G$-space, and in particular $\dim M_L=\dim {}_L M= |G|$.
\item If $M$ is a free skew field, then $G$ is a group of outer automorphisms.
\end{enumerate}
\end{lemma}

\begin{proof}
(1) By Lemma \ref{l:dedekind}, the automorphisms $g\in G$ are $L$-independent. In particular, their sum is a nonzero endomorphism of $M$, so there exists $y\in M$ such that $\overline{y}=\sum_{g\in G}y^g\neq0$. 
Let $d=|G|$, and let $f_1,\dots,f_{d+1}\in M$ be arbitrary. 
The set of functions from $G$ to $M$ is a free right $M$-module of rank $d$.
Define functions $F_i:G\to M$ by $F_i(g)=f_i^g$. Then $F_1,\dots,F_{d+1}$ are dependent.
After reordering we may assume that  $F_{d+1}=\sum_{i=1}^d F_ia_i$
for some $a_1,a_2,\dots,a_\ell\in M$.
That is, $f_{d+1}^gy=\sum_{i=1}^d f_i^ga_iy$ for every $g\in G$.
Therefore $f_{d+1}y^{g^{-1}}=\sum_{i=1}^\ell f_i (a_iy)^{g^{-1}}$ for all $g\in G$, and so
$$f_{d+1}\overline{y}=\sum_{i=1}^d f_i \overline{a}_i,\qquad \text{where}\quad \overline{a}_i=\sum_{g\in G}(a_iy)^g\in L.$$
Since $\overline{y}\in L$ is nonzero, we conclude that $f_1,\dots,f_{d+1}$ are linearly dependent over $L$. Hence, $\dim M_L\le d$. The inequality $\dim {}_LM\le d$ is proved analogously.

(2) This is a consequence of the normal basis theorem for finite groups of outer automorphisms \cite[Section 4.1]{Dre}, which states that there is $x\in M$ such that $\{x^g:g\in G\}$ has size $|G|$ and is a basis of $M$ as a right (or left) $L$-module. Such an $x$ then gives rise to an isomorphism $\kk[G]\otimes L\to M$ given by $g\otimes\ell \mapsto x^g\ell$. 

(3) By \cite[Exercise 6.2.4]{Coh1}, no element in $M\setminus\kk$ is algebraic over $\kk$. Also, $M$ is commutative if its rank is 1, and otherwise has center equal to $\kk$ by \cite[Corollary 7.9.7]{Coh2}.
Consequently, an inner automorphism of $M$ is either the identity or has infinite order. Therefore, the identity is the only inner automorphism of $M$ in $G$. 
\end{proof}

In general, the inequality in Lemma \ref{l:fixed}(1) is strict (see the example after \cite[Theorem~3.3.7]{Coh1}). 

\subsection{Linear group actions}
In this section we will assume that $G$ is a finite group such that $\kar(\kk)\nmid |G|$. This implies that $G$ is linearly reductive.
We summarize some consequences of this.
Whenever $X$ is a subrepresentation of a finite dimensional representation $Y$, it has a $G$-stable complement $X^\perp$, so $Y=X\oplus X^\perp$. 
Every representation of $G$ is a direct sum of irreducible representations.
A linear map between representations $\phi:X\to Y$ is $G$-equivariant if $\phi(x^g)=\phi(x)^g$ for all $x\in X$ and $g\in G$. In that case we have $\phi(X^G)=Y^G$.

We can strengthen Theorem \ref{theo:FinitelyGenerated} for skew subfields of invariants for finite linear groups. We start by refining Proposition \ref{prop:Ws}.

\begin{proposition}\label{prop:Ws_lin}
Let $G$ be a finite outer group of automorphisms on a skew field $M$, and $\kar\kk\nmid |G|$. Suppose there is an $m$-dimensional $G$-stable $\kk$-subspace of $M$ that generates $M$ over $\kk$. 
Then, there exists a $G$-stable $\kk$-subspace $W$ of the subalgebra in $M$ generated by $V$, with the following properties:
\begin{enumerate}
    \item $WM^G=M$ and $\dim W=\dim M_{M^G}$;
    \item $\kk\subseteq W\subseteq \kk+VW$;
    \item $W^\star\cong \kk[G]$ as $G$-spaces.
\end{enumerate}
\end{proposition}

\begin{proof}
Let $L=M^G$; by Lemma \ref{l:fixed} we have $\dim M_L=|G|$. Let $L[G]$ be the group algebra of $G$ over $L$, where elements of $L$ commute with group elements in $G$. Since $\kar\kk\nmid |G|$, the group $G$ is linearly reductive and $L[G]$ is semisimple.

We follow the proof of Proposition \ref{prop:Ws} to construct a space $W$ with $WL=M$, $\dim W=\dim M_L$ and $\kk\subseteq W\subseteq \kk+VW$. However, we claim that we can choose $W$ such that it is stable under the action of $G$. Since $G$ is linearly reductive, any $G$-stable subspace in a representation of $G$ over $\kk$ has a complement. In the proof of Proposition \ref{prop:Ws}, we construct subspaces $W_i$ for $i\geq 0$. We claim that we can choose each $W_i$ to be $G$-stable. The space  $W_0=\kk$ it is clearly $G$-stable.
By induction we may assume that $W_i$ has been constructed, is $G$-stable and 
$V_+^iL+VW_{i}L=V_+^{i+1}L$. 
Since $VW_i$ is also $G$-stable, we can write $VW_{i}=A_1\oplus \cdots \oplus A_r$ where $A_1,\dots,A_r$ are irreducible representations of $G$. 
Now $A_iL\cong A_i\otimes L$ is a simple right $L[G]$-module for all $i$. 
By Lemma \ref{lem:complement}, we can take some subset 
$I\subseteq \{1,\dots,r\}$ such that 
$V_+^iL\oplus \bigoplus_{j\in I}A_jL=V_+^{i+1}L$. Now take $W_{i+1}=\bigoplus_{j\in I}A_j$. 
Then we have $V_+^iL\oplus W_{i+1}L=V_{+}^{i+1}L$ and
$W_{i+1}L=\bigoplus_{j=1}^r A_jL\cong \bigoplus_{j\in I} A_j\otimes L$, so $\dim (W_{i+1}L)_L=\sum_{j\in I} \dim (A_j\otimes L)_L
=\sum_{i\in I} \dim(A_j)=\dim(W_{i+1})$.
Now $W_{i+1}$ is $G$-stable because it is a right $L[G]$-module.

By the construction, $W=\bigoplus_i W_i$ is a $G$-stable subspace of the subalgebra in $M$ generated by $V$, and satisfies (1) and (2) as in Proposition \ref{prop:Ws}. 
Since $W\otimes L\cong M\cong \kk[G]\otimes L$ as $G$-spaces and $L$-modules by Lemma \ref{l:fixed}, it follows that $W\cong \kk[G]$. Therefore, $W^\star \cong \kk[G]$, and (3) holds.
\end{proof}

\begin{theorem}\label{theo:LinFinitelyGenerated}
Let $G$ be a finite group of outer automorphisms on a skew field $M$, and $\kar\kk\nmid |G|$. 
If $V$ is an $m$-dimensional $G$-stable $\kk$-subspace of $M$ that generates $M$ over $\kk$, 
then $M^G$ is generated by at most $|G|(m-1)+1$ generators.
\end{theorem}
\begin{proof}
We follow the proof of Theorem \ref{theo:FinitelyGenerated} with $L=M^G$ by taking the group action into account. Let $W$ be as in Proposition \ref{prop:Ws_lin}.
Let $d=\dim W=\dim M_L$; by Lemma \ref{l:fixed} we have $d=|G|$. 
Let $\rho:M\to \End(W)\otimes L$, $\rho':M\otimes W\to W\otimes L$ and $\rho'':M\otimes W\otimes W^\star\to L$ be as in the proof of Theorem \ref{theo:FinitelyGenerated}. 
The subspace $W\cap VW$ is $G$-stable. We can find a $G$-stable subspace $T\subseteq V\otimes W$ with $\rho'(T)=(W\cap VW)\otimes \kk$. We can also choose a $G$-stable complement $T^\perp$ of $T$ in $V\otimes W$. Since $G$ acts by automorphisms on $M$ and trivially on $L=M^G$, the map $\rho'':V\otimes W\otimes W^\star \to L$ is $G$-invariant. 
In particular, 
 we have $Z=Z^G=\rho''(T^{\perp}\otimes W^\star)^G=\rho''((T^\perp\otimes W^\star)^G)$. 
Since $W^\star\cong \kk[G]$ and $(U\otimes \kk[G])^G\cong U$ for every $G$-space $U$, it follows that 
$\dim Z\leq \dim (T^\perp\otimes W^\star)^G= \dim (T^\perp\otimes \kk[G])^G= \dim T^{\perp}\leq md-(d-1)$.
\end{proof}

\section{Skew subfields of free skew fields}\label{sec:free}

We can enhance the results of Section \ref{sec:general} in the context of free skew fields. Throughout the section let $M=\rx{x_1,\dots,x_m}$ be the free skew field over $\kk$ generated by free variables $x_1,\dots,x_m$.

\subsection{Fundamental relations}

Let $L$ be a skew subfield of $M$ with $\dim M_L=d$. 
Let $W\subseteq M$, $\rho: M\to \End(W)\otimes L$ and $Z\subseteq L$ be as in Proposition \ref{prop:Ws} and the proof of Theorem \ref{theo:FinitelyGenerated}. 
Let us choose a basis of $W$ that starts with 1, and let $X_i\in L^{d\times d}$ be the matrix representing $\rho(x_i)$ relative to the chosen basis. Then $\kk 1+Z$ is the $\kk$-linear span of the entries of $X_1,\dots,X_m$. 

Let $z_1,\dots,z_t$ span $Z$. Then, there exist $d\times d$ matrices $A_{i0},\dots,A_{it}$ over $\kk$ such that 
\begin{equation}\label{e:Xi}
X_i=A_{i0}+z_1A_{i1}+\dots +z_tA_{it}.
\end{equation}
Let $\fz_1,\dots,\fz_t$ be auxiliary free variables, and denote 
\begin{equation}\label{e:fXi}
\fX_i := A_{i0}+\fz_1A_{i1}+\dots +\fz_tA_{it},
\end{equation}
viewed as a $d\times d$ matrix over $\px{\fz_1,\dots,\fz_t}\subset\rx{\fz_1,\dots,\fz_t}$. 
If $z_1,\dots,z_t$ form a basis of $Z$, the matrices $\fX_i$ are uniquely determined by this basis (which is typically the case one is interested in, but for the sake of a later use we only assume that $z_1,\dots,z_t$ span $Z$).
Using matrices \eqref{e:fXi}, we provide an explicit finite collection of relations for $L$ that are in a certain sense fundamental for $L$.

\begin{proposition}\label{p:freetest} 
With the notation as above, denote
\begin{align*}
\cR&=\Big\{ z_j(\fX_1,\dots,\fX_m)_{11}-\fz_j\colon j=1,\dots,t \Big\} \\
&\ \quad\cup
\Big\{ (z_j(\fX_1,\dots,\fX_m)_{\ell1}\colon j=1,\dots,t,\ \ell=2,\dots,d \Big\}.
\end{align*}
Let $D$ be a skew field generated by $a_1,\dots,a_t$ such that $r(a_1,\dots,a_t)=0$ for all $r\in\cR$. If there is a local homomorphism from $D$ to $L$ such that $a_j\mapsto z_j$ for $j=1,\dots,t$, then $D\cong L$.

In particular, if $\cR=\{0\}$, then $L$ is the free skew field of rank $t$.
\end{proposition}

\begin{proof}
Note that the first column of $z_j(X_1,\dots,X_m)$ equals $z_j e_1$, where $e_1$ is the first standard basis vector in $\kk^d$. Indeed, $\rho$ is a homomorphism and thus
$z_j(\rho(x_1),\dots,\rho(x_m))=\rho(z_j)$, and $z_j \in L$ as an element of the free $L$-module $M=WL$ is simply the $z_j$-multiple of the first basis vector.

Next, we claim that the homomorphism $\phi:\px{x_1,\dots,x_m}\to D^{d\times d}$ defined by $x_i\mapsto \fX_i(a_1,\dots,a_t)$ extends to an embedding $\Phi:M\to D^{d\times d}$. 
Since $M$ is the universal skew field of fractions of $\px{x_1,\dots,x_m}$, its elements are entries of inverses of matrices over $\px{x_1,\dots,x_m}$ by \cite[Theorem 4.2.1]{Coh1}. 
Thus, it suffices to see the following for a square matrix $A$ over $\px{x_1,\dots,x_m}$: if $A$ is invertible over $M$, then the entry-wise application of $\phi$ maps it to an invertible matrix over $D$. 
Since $A$ is invertible over $M$ and $\rho:M\to L^{d\times d}$ is an embedding of a skew field, the matrix $A(X_1,\dots,X_m)$ is invertible over $L$. Consider the matrix $B=A(\fX_1,\dots,\fX_m)$ over $\px{\fz_1,\dots,\fz_t}$. 
Since $B(z_1,\dots,z_t)=A(X_1,\dots,X_m)$ is invertible over $L$, the matrix $B(a_1,\dots,a_t)$ is invertible over $D$ by Lemma \ref{l:univ}, as desired.

Finally, consider the restricted homomorphism of $\Phi|_L:L\to D^{d\times d}$. For every $j=1,\dots,t$, the first column of $\Phi(z_j)$ equals $a_je_1$ by the definition of $\cR$. Therefore, the map $L\to D$ given by $u\mapsto \Phi(u)_{11}$ is a surjective homomorphism of skew fields, and thus an isomorphism.
\end{proof}

More generally, the conclusion of Proposition \ref{p:freetest} holds even if the entries of $X_i$ are not affine in the generators of $L$, but rational in the generators of $L$.
Note that Proposition~\ref{p:freetest} reduces freeness of generators $z_1,\dots,z_t$ to noncommutative rational identity testing, which can be done efficiently \cite{Gar,DM,Iva}. 
In general, freeness translates to checking linear independence over the enveloping algebra $M\otimes M^\op$ via universal derivation modules (see \cite[Section 5.8]{Coh1} or \cite[Section 10]{Sch}); however, this does not evidently yield an effective algorithm for deciding freeness.

\subsection{Invariants in free skew fields}

Theorem \ref{theo:FinitelyGenerated} and Lemma \ref{l:fixed} imply the following.

\begin{corollary}\label{c:nonlin}
If $G$ is a finite group of automorphisms on $\rx{x_1,\dots,x_m}$, 
then the skew subfield $\rx{x_1,\dots,x_m}^G$ is generated over $\kk$ by at most $|G|^2(m-1)+|G|$ elements.
\end{corollary}

Let us further simplify the presentation of generators and their relations from Corollary~\ref{c:nonlin}. Denote $M=\rx{x_1,\dots,x_m}$ and $L=M^G$. Let $1=w_1,w_2,\dots,w_d$ be a basis of the vector space $W$ as in Proposition \ref{prop:Ws}; note that $d=|G|$ by Lemma \ref{l:fixed}. 
If $X_i\in L^{d\times d}$ is the matrix representing $\rho(x_i)$ relative to this basis, then
\begin{equation}\label{e:freeqs0}
x_iw_j = \sum_{k=1}^d w_k\cdot (X_i)_{kj} 
\quad \text{for }i=1,\dots,m,\ j=1,\dots,d.
\end{equation}
Applying an automorphism $g\in G$ to \eqref{e:freeqs0} gives
\begin{equation}\label{e:freeqs0g}
x_i^gw_j^g = \sum_{k=1}^d w_k^g\cdot (X_i)_{kj} 
\quad \text{for all $i$, $j$ and $g\in G$.}
\end{equation}

For $i=1,\dots,m$ let $\sD_i$ be the diagonal $d\times d$ matrix indexed by $g\in G$ whose $g$\textsuperscript{th} diagonal entry equals $x_i^g$. Consider the $d\times d$ matrix $\sW=(w_j^g)_{g,j}$, whose rows are indexed by $g\in G$, and columns are indexed by $j=1,\dots,d$. 

\begin{lemma}\label{l:Winvert}
The matrix $\sW$ is invertible over $M$. 
\end{lemma}

\begin{proof}
By Lemma \ref{l:dedekind}, there exists $y\in M$ such that $\sum_{g\in G}y^g\neq0$. Suppose $\sW$ is not invertible over $M$. Then there is a vector $u=(u_j)_j$ over $M$ such that $\sW u=0$ and $u_{j_0}=y$ for some $j_0$. Since $\sW u=0$ implies $\sum_jw_j^gu_j=0$ for all $g\in G$,
we  get 
$\sum_jw_ju_j^{g^{-1}}=\sum_j(w_j^gu_i)^{g^{-1}}=0$. Summing over all $g\in G$ gives $\sum_jw_j\overline{u}_j=0$ where $\overline{u}_j=\sum_{g\in G} u_j^g$. Since $\overline{u}_j\in L$ for all $j$, and $\overline{u}_{j_0}=\sum_{g\in G}y^g\neq0$, it follows that $w_1,\dots,w_d$ are linearly dependent over $L$, a contradiction.
\end{proof}

Since $\sW$ is invertible, \eqref{e:freeqs0g} can be rewritten as
\begin{equation}\label{e:freeqs2}
\sD_i \sW=\sW X_i\quad\mbox{or}\quad
X_i=\sW^{-1}\sD_i\sW \qquad \text{for } i=1,\dots,m,
\end{equation}
which is a convenient way of obtaining the generators $(X_i)_{kj}$ of $L$. Furthermore, Proposition~\ref{p:freetest} shows that 
\begin{equation}\label{e:fundamentalrelations}
z_j(\fX_1,\dots,\fX_m)e_1-\fz_je_1 \quad
\text{for }
j=1,\dots,t,\ \ell=1,\dots,d
\end{equation}
(where $e_1$ is the first standard unit vector) 
vanish (are constantly zero) when evaluated on $z_1,\dots,z_t$ forming a basis of $Z$, and are the fundamental relations satisfied by these generators.

There is a convenient way to reformulate the relations \eqref{e:fundamentalrelations} using a block matrix equation that does not involve choosing a basis of $Z$. Let
$X_i(\fX_1,\dots,\fX_m)$ be the block matrix obtained by replacing $x_k$ by $\fX_k$ for $i=1,\dots,m$ in the matrix $X_i$.
Using \eqref{e:Xi}, this is equal to
$$
A_{i0}\otimes I+A_{i1}\otimes z_1(\fX_1,\dots,\fX_m)+\cdots+A_{it}\otimes z_t(\fX_1,\dots,\fX_m),
$$
where $\otimes$ is the Kronecker tensor product of matrices.
After multiplying on the right by $I\otimes e_1$, we get
\begin{align*}
&\ X_i(\fX_1,\dots,\fX_m)(I\otimes e_1)-\fX_i\otimes e_1\\
=&\ A_{i0}\otimes e_1+A_{i1}\otimes z_1(\fX_1,\dots,\fX_m)e_1+\cdots+A_{it}\otimes z_t(\fX_1,\dots,\fX_m)e_1 \\
&\quad -(A_{i0}+\fz_1 A_{i1}+\cdots+\fz_tA_{it})\otimes e_1\\
=&\ A_{i1}\otimes (z_1(\fX_1,\dots,\fX_m)e_1-\fz_1e_1)+\cdots+
A_{it}\otimes (z_t(\fX_1,\dots,\fX_m)e_1-\fz_te_1)
\end{align*}
by \eqref{e:fXi}. Combining this equation with \eqref{e:freeqs2} shows that vanishing of \eqref{e:fundamentalrelations} implies vanishing of
\begin{equation}\label{e:freeqsM}
\sD_i(\fX_1,\dots,\fX_m)\sW(\fX_1,\dots,\fX_m) (I\otimes e_1) - \sW(\fX_1,\dots,\fX_m) (\fX_i\otimes e_1) \quad \text{for }i=1,\dots,m.
\end{equation}
But because the span of the entries of the $X_i$'s equals the span of $z_1,\dots,z_t$, vanishing of \eqref{e:freeqsM} also implies vanishing of \eqref{e:fundamentalrelations}.
The benefit of \eqref{e:freeqsM} is that these equations do not depend on the choice of a basis of $Z$ (as opposed to the relations in Proposition \ref{p:freetest}).

\subsubsection{Bivariate invariants of involutions}\label{ex:Z2}

Let us demonstrate the findings of this section on faithful actions of $\Z_2$ on $\rx{x,y}$. Such an action is given by an involution $\phi:\rx{x,y}\to\rx{x,y}$, a $\kk$-linear automorphism of order two. Without loss of generality let $x^\phi\neq x$. 
Then, the subspace $W=\operatorname{span}_{\kk}\{1,x\}$ satisfies the conclusions of Proposition \ref{prop:Ws}. The corresponding matrices $\sD_i$ and $\sW$ are
$$\sD_1=\begin{pmatrix}x&0\\0&x^\phi\end{pmatrix},\quad
\sD_2=\begin{pmatrix}y&0\\0&y^\phi\end{pmatrix},\quad
\sW=\begin{pmatrix}1& x\\ 1& x^\phi\end{pmatrix}.$$
Next, we calculate
$$\sW^{-1}=
\begin{pmatrix}1+x(x^\phi-x)^{-1}x^\phi& -x(x^\phi-x)^{-1}\\ -(x^\phi-x)^{-1}x^\phi& (x^\phi-x)^{-1}\end{pmatrix}.
$$
The matrices $X=\sW^{-1}\sD_1\sW$ and $Y=\sW^{-1}\sD_2\sW$ as in \eqref{e:freeqs2} are then equal to
\begin{equation}\label{e:XandY}
X=\begin{pmatrix}
    0 & z_1\\
    1 & z_2
\end{pmatrix}\quad\mbox{and}\quad Y=\begin{pmatrix}
    z_3 & z_5\\z_4 & z_6
\end{pmatrix},\quad \mbox{where}
\end{equation}
\begin{equation}\label{e:z2general}
\def\arraystretch{1.4}
\begin{array}{rcl}
    z_1 & = & x^\phi\big(x^\phi-x\big)^{-1}x^2-x\big(x^\phi-x\big)^{-1}(x^\phi)^2\\
    z_2 & = & \big(x^\phi-x\big)^{-1}\big((x^\phi)^2-x^2\big)\\
    z_3 & = &  x^\phi\big(x^\phi-x\big)^{-1}y-x\big(x^\phi-x\big)^{-1}y^\phi\\
    z_4 & = & \big(x^\phi-x\big)^{-1}\big(y^\phi-y\big)\\
    z_5 & = &  x^\phi\big(x^\phi-x\big)^{-1}yx-x\big(x^\phi-x\big)^{-1}y^\phi x^\phi\\
    z_6 & = & \big(x^\phi-x\big)^{-1}\big(y^\phi x^\phi-yx\big).\end{array}
\end{equation}
The six not-obviously-scalar entries $z_1,\dots,z_6$ of $X,Y$ generate $\rx{x,y}^{\Z_2}$ by Corollary~\ref{c:nonlin}. By Proposition \ref{p:freetest}, all relations between these generators are captured by the $6\cdot 2=12$ equations as in \eqref{e:freeqsM}.
In Section \ref{sec:nonfree} below we further investigate these equations for a specific nonlinear involution on $\rx{x,y}$, where we see that $\rx{x,y}^{\Z_2}$ is not necessarily free, nor can it be generated by three generators in general. But first, we analyze $\rx{x,y}^{\Z_2}$ for a linear action of $\Z_2$ on $\rx{x,y}$; that is, we assume that $x^\phi,y^\phi$ are linear combinations of $x,y$.
If $\kar\kk\neq 2$, then there are, up to isomorphism, only the following two examples of nontrivial $2$-dimensional representations of $\Z_2$:

\begin{example}
Let $\kar\kk\neq2$; and assume that $x^\phi=-x$ and $y^\phi=y$. Then, the  equations~\eqref{e:XandY} and \eqref{e:z2general} simplify to
$$
X=\begin{pmatrix}
0& x^2 
\\1&0
\end{pmatrix},\quad
Y=\begin{pmatrix}
y & 0 \\ 
0 & x^{-1}yx
\end{pmatrix}.
$$
In this case, $\rx{x,y}^{\Z_2}$ is generated by three generators: $x^2,y,x^{-1}yx$ or $x^2,yx,x^{-1}y$. Furthermore, these generators are free; since $\Z_2$ is abelian and $\kar\kk\neq2$, this holds by \cite[Theorem 4.1]{KPPV}, or by Theorem \ref{t:linfree} below for general finite linear groups. Alternatively, one can directly check that the six relations (two per generator of $\rx{x,y}^{\Z_2}$) in Proposition \ref{p:freetest} are trivial; for example, if $y^\phi=y$, the generators are $z_1=x^2,z_2=y,x_3=x^{-1}yx$, and the first columns of
\begin{align*}
&\begin{pmatrix}
0& \fz_1 
\\1&0
\end{pmatrix}^2=\begin{pmatrix}\fz_1&0\\0&\fz_1\end{pmatrix},\quad
\begin{pmatrix}
\fz_2& 0 
\\0&\fz_3
\end{pmatrix}=\begin{pmatrix}\fz_2&0\\0&\fz_3\end{pmatrix},\\
&\begin{pmatrix}
0& \fz_1 
\\1&0
\end{pmatrix}^{-1}\begin{pmatrix}
\fz_2& 0 
\\0&\fz_3
\end{pmatrix}\begin{pmatrix}
0& \fz_1 
\\1&0
\end{pmatrix}=\begin{pmatrix}\fz_3&0\\0&\fz_1^{-1}\fz_2\fz_1\end{pmatrix}
\end{align*}
equal $\fz_j$ times $e_1$, implying that there are no relations between $z_1,z_2,z_3$.
\end{example}

\begin{example}
Let $\kar\kk\neq2$ and assume that $x^\phi=-x$ and $y^\phi=- y$ after a linear change of coordinates. Then, from~\eqref{e:XandY} and \eqref{e:z2general} follows that
$$
X=\begin{pmatrix}
0& x^2 
\\1&0
\end{pmatrix},\quad
Y=\begin{pmatrix}
0 & yx\\ 
x^{-1}y & 0
\end{pmatrix}.
$$
In this case, $\rx{x,y}^{\Z_2}$ is generated by $x^2,yx,x^{-1}y$. The reader may verify, as in the previous example, that these are free generators.
\end{example}

If $\kar\kk=2$, then there is a unique nontrivial $2$-dimensional representation of $\Z_2$ up to isomorphism.

\begin{example}
Let $\kar\kk=2$ and can assume that $x^\phi=y$ and $y^\phi=x$. Then, it follows from \eqref{e:XandY} and~\eqref{e:z2general} that
$$
X=\begin{pmatrix}
0& x(x+y)^{-1}y(x+y) 
\\1& (x+y)^{-1}(x^2+y^2)
\end{pmatrix},\quad
Y=\begin{pmatrix}
x+y & xy+yx+x(x+y)^{-1}y(x+y)\\ 
1 & (x+y)^{-1}(xy+yx)
\end{pmatrix}.
$$
The entries of these two matrices generate $\rx{x,y}^{\Z_2}$; a short calculation shows that the diagonal entries of $Y$ and the top right entry of $X$ generate the remaining entries. Thus, $\rx{x,y}^{\Z_2}$ is generated by $z_1=x+y,z_2=xy+yx,z_3=x^{-1}+y^{-1}$ (note that $z_3$ is the inverse of $x(x+y)^{-1}y$), and these are free by Proposition \ref{p:freetest} (more precisely, the sentence after it) because
\begin{align*}
&\left(\begin{pmatrix}0& \fz_3^{-1}\fz_1 \\1&\fz_1+\fz_1^{-1}\fz_2\end{pmatrix}+
\begin{pmatrix} \fz_1& \fz_2+\fz_3^{-1}\fz_1 
\\1&\fz_1^{-1}\fz_2
\end{pmatrix}\right)\begin{pmatrix}1\\0\end{pmatrix}=\begin{pmatrix}\fz_1\\0\end{pmatrix}, \\
&\left(\begin{pmatrix}0& \fz_3^{-1}\fz_1 \\1&\fz_1+\fz_1^{-1}\fz_2\end{pmatrix}\begin{pmatrix} \fz_1& \fz_2+\fz_3^{-1}\fz_1 \\1&\fz_1^{-1}\fz_2\end{pmatrix}+
\begin{pmatrix} \fz_1& \fz_2+\fz_3^{-1}\fz_1 \\1&\fz_1^{-1}\fz_2\end{pmatrix}\begin{pmatrix}0& \fz_3^{-1}\fz_1 \\1&\fz_1+\fz_1^{-1}\fz_2\end{pmatrix}\right)
\begin{pmatrix}1\\0\end{pmatrix}=\begin{pmatrix}\fz_2\\0\end{pmatrix}, \\
&\left(\begin{pmatrix}0& \fz_3^{-1}\fz_1 \\1&\fz_1+\fz_1^{-1}\fz_2\end{pmatrix}^{-1}+
\begin{pmatrix} \fz_1& \fz_2+\fz_3^{-1}\fz_1 
\\1&\fz_1^{-1}\fz_2
\end{pmatrix}^{-1}\right)\begin{pmatrix}1\\0\end{pmatrix}=\begin{pmatrix}\fz_3\\0\end{pmatrix}.
\end{align*}
\end{example}

\subsection{Linear group actions on free skew fields}\label{ss:free}

In this section we show that the generators $z_1,\dots,z_t$ obtained in the proof of Theorem \ref{theo:LinFinitelyGenerated} are free when $G$ is a finite group of linear automorphisms on $M=\rx{x_1,\dots,x_m}$. This is achieved using Proposition \ref{p:freetest}; the following construction will help us show that the relations in Proposition \ref{p:freetest} are trivial (i.e., the zero functions in $\fz_1,\dots,\fz_t$). 

Given a vector space $U$ over $\kk$, let $\px{U}$ be the tensor algebra over $U$ (which is isomorphic to the free algebra of rank $\dim U$, justifying the notation). 
Let $V=\operatorname{span}_\kk\{x_1,\dots,x_m\}$. Suppose $G$ is a finite group of automorphisms on $M$ with $\kar\kk\nmid |G|$, and $V$ is $G$-stable. Let $W\subset\px{V}$ be as in Proposition \ref{prop:Ws_lin}. 
We endow $V\otimes W$ and $V\otimes W\otimes W^\star$ with the diagonal action of $G$. 
Choose a $G$-stable complement $(W\cap VW)^\perp$ of $W\cap VW$ in $VW\equiv V\otimes W$ (this identification is merited because $V$ and $W$ are subspaces of a free algebra, and $V$ is homogeneous).
Denote the vector space $\fZ=\left((W\cap VW)^\perp\otimes W^\star\right)^G$, and consider the surjective map
$$\pi'':V\otimes W\otimes W^\star 
= \big((W\cap VW)\otimes W^\star\big)\oplus \big((W\cap VW)^\perp\otimes W^\star\big)
\xrightarrow{\tr \oplus \cR} \kk\oplus \fZ,
$$
where $\tr:W\otimes W^\star\to\kk$ is the trace $\tr(w\otimes \varphi)=\varphi(w)$ and $\cR: (W\cap VW)^\perp\otimes W^\star\to \fZ$ is the Reynolds operator $\cR(u)=\frac{1}{|G|}\sum_{g\in G}u^g$. 
Then $\pi''$ is a $G$-invariant linear map.
We can view $\pi''$ as a $G$-equivariant map in $3$ different ways:
\begin{eqnarray*}
\pi:V & \to & W^\star\otimes W\otimes (\kk\oplus \fZ)\\
\pi':V\otimes W & \to & W\otimes (\kk\oplus\fZ)\\
\pi'':V\otimes W\otimes W^\star & \to & \kk\oplus \fZ.
\end{eqnarray*}
Now 
$$\pi:V\to W^\star\otimes W \otimes (\kk\oplus \fZ)\cong  \End(W) \otimes (\kk\oplus \fZ)
\subseteq \End(W)\otimes \px{\fZ}$$ extends to a $G$-equivariant algebra homomorphism
$$
\Pi:\px{V}\to \End(W)\otimes \px{\fZ}.
$$

\begin{lemma}\label{l:formal}
Assume the notation from the preceding paragraph.
Let $\{w_j\}_j$ be a basis of $W$, and let $\{w_j^\star\}_j$ be its dual basis of $W^\star$. Then, 
\begin{equation}\label{e:goal}
\Pi\big(vw\big)(1\otimes 1) = \sum_{k=1}^d \Pi(w_k) \big(1\otimes\pi''(v\otimes w\otimes w_k^\star)\big)
\end{equation}
for all $v\in V$ and $w\in W$ (where $\Pi(vw),\Pi(w_k)$ act on $W\otimes \px{\fZ}$).
\end{lemma}

\begin{proof}
First, we show that
\begin{equation}\label{e:piw}
\Pi(w)(1\otimes 1) =w\otimes 1 \quad \text{for }w\in W.
\end{equation}
We prove \eqref{e:piw} by induction on $\deg w$, using $W\subseteq \kk+VW$. Clearly, \eqref{e:piw} holds for $w\in \kk$. Now let $w=\alpha+\sum_j \ell_j w_j$ where $\ell_j\in V$ and $w_j\in W$ satisfy $\Pi(w_j)(1\otimes 1)=w_j\otimes 1$. Then
\begin{align*}
\Pi(w)(1\otimes 1)
&=\alpha\otimes1+\sum_j\Pi(\ell_j)\Pi(w_j)(1\otimes 1) 
=\alpha\otimes1+\sum_j\pi(\ell_j)(w_j\otimes 1) \\
&=\alpha\otimes1+\sum_j\pi'(\ell_j\otimes w_j) \in W\otimes (\kk\oplus \fZ),
\end{align*}
and hence
\begin{align*}
(w_k^\star\otimes \id)\big(\Pi(w)(1\otimes 1)\big)
&=\alpha w_k^\star(1)+\sum_j (w_k^\star\otimes \id)\big(\pi'(\ell_j\otimes w_j)\big) \\
&=\alpha w_k^\star(1)+\sum_j \pi''(\ell_j\otimes w_j\otimes w_k^\star) \\
&=\pi''\left(\left(\alpha\otimes 1+\sum_j\ell_j\otimes w_j\right)\otimes w_k^\star\right)
=w_k^\star(w)
\end{align*}
for all $k$, which implies $\Pi(w)(1\otimes 1)=w\otimes1$.
Now, consider both sides of \eqref{e:goal}. Using \eqref{e:piw}, they simplify to
\begin{align}
\label{e:lhs} &\Pi(vw)(1\otimes 1)
=\Pi(v)\Pi(w)(1\otimes 1)
=\pi(v)(w\otimes 1)=\pi'(v\otimes w)\in W\otimes (\kk \oplus \fZ), \\
\label{e:rhs} &\sum_k \Pi(w_k) \big(1\otimes\pi''(v\otimes w\otimes w_k^\star)\big)
= \sum_k w_k\otimes \pi''(v\otimes w\otimes w_k^\star) \in W\otimes (\kk\oplus \fZ).
\end{align}
Applying $w_{k'}^\star\otimes \id$ to \eqref{e:lhs} results in
$$(w_{k'}^\star\otimes \id)\big(\pi'(v\otimes w)\big)
=\pi''(v\otimes w\otimes w_{k'}^\star),$$
while applying it to \eqref{e:rhs} results in
$$\sum_k w_{k'}^\star(w_k) \pi''(v\otimes w\otimes w_k^\star)=\pi''(v\otimes w\otimes w_{k'}^\star).$$
Since this is true for all $k'$, it follows that \eqref{e:lhs} and \eqref{e:rhs} are equal, and so \eqref{e:goal} holds.
\end{proof}

The system \eqref{e:goal} of polynomial identities in $\px{\fZ}$ is the last ingredient for the proof of the next main result.

\begin{theorem}\label{t:linfree}
Let $G$ be a finite group of linear automorphisms on $\rx{x_1,\dots,x_m}$, and $\kar\kk\nmid |G|$. 
Then, $\rx{x_1,\dots,x_m}^G$ is the free skew field of rank $|G|(m-1)+1$.
\end{theorem}

\begin{proof}
Let $M=\rx{x_1,\dots,x_m}$ and $V=\operatorname{span}_\kk\{x_1,\dots,x_m\}$. Note that $G$ consists of outer automorphisms. Let $W$ be as in Proposition \ref{prop:Ws_lin}, and $d=\dim W=\dim M_{M^G} = |G|$.

We mimic the proof of Theorem \ref{theo:LinFinitelyGenerated}; let $\rho:M\to \End(W)\otimes L$ and $\rho'':M\otimes W\otimes W^\star\to L$ be as therein. 
Since $W$ is a subspace of the free algebra $\px{V}$ and $W\subseteq \kk+VW$, it follows that the $G$-stable subspace $W\cap VW\subseteq VW$ has dimension $d-1$. 
Since $VW=V\otimes W$ in the tensor algebra $\px{V}$, we can view $W\cap VW$ as a subspace of $V\otimes W$ (that plays the role of $T$ from the proofs of Theorems \ref{theo:FinitelyGenerated} and \ref{theo:LinFinitelyGenerated}).
Choose a $G$-stable complement $(W\cap VW)^\perp$ of $W\cap VW$ in $V\otimes W$, and denote $\fZ=((W\cap VW)^\perp\otimes W^\star)^G\subseteq V\otimes W\otimes W^\star$. 
Since $W^\star\cong \kk[G]$, we have $\dim \fZ=\dim\big((W\cap VW)^\perp\otimes W^\star\big)^G=\dim (W\cap VW)^\perp$.
The map $\rho'':V\otimes W\otimes W^\star \to L$ is $G$-invariant, and consequently $Z=\rho''((W\cap VW)^\perp\otimes W^\star)=\rho''(\fZ)$. Recall that $L$ is then generated by a basis of $Z$ by the proof of Theorem \ref{theo:FinitelyGenerated}. 

Let $\pi'':V\otimes W\otimes W^\star\to\kk\oplus \fZ$ 
and $\Pi:\px{V}\to\End(W)\otimes\px{\fZ}$ be as in Lemma~\ref{l:formal}. Fix a basis $1=w_1,\dots,w_d$ of $W$, and its dual basis $w_1^\star,\dots,w_d^\star$ in $W^\star$. Then,
\begin{equation}\label{e:freeqs1}
\Pi\big(x_i^gw_j^g\big)(1\otimes 1) - \sum_{k=1}^d \Pi(w_k^g) \big(1\otimes\pi''(v\otimes w_j\otimes w_k^\star)\big)=0
\quad \text{for all }i,j,g
\end{equation}
by \eqref{e:goal} because $\Pi$ is $G$-equivariant. 
By the construction of $\rho,\rho''$ and $\Pi,\pi''$, the diagrams 
\begin{equation}
\label{e:diagram}
\begin{tikzcd}
V\otimes W\otimes W^\star \arrow[r, "\pi''"] \arrow[rd, "\rho''"'] 
& \kk\oplus \fZ \arrow[d, "\rho''"] \\
& \kk \oplus Z
\end{tikzcd} \qquad \begin{tikzcd}
\px{V} \arrow[r, "\Pi"] \arrow[rd, "\rho"'] 
& \End(W)\otimes\px{\fZ} \arrow[d, ""] \\
& \End(W)\otimes L
\end{tikzcd}
\end{equation}
commute; here, the vertical map in the second diagram in \eqref{e:diagram} is the unique algebra homomorphism that restricts to the identity on $\End(W)$ and whose restriction to $\fZ$ agrees with $\rho''$. 
Let $\fz_1,\dots,\fz_t$ form a basis of $\fZ$; then $z_1=\rho''(\fz_1),\dots,z_t=\rho''(\fz_t)$ span $Z$, and our goal is to show that they are free in $M$.
Since $\fz_1,\dots,\fz_t$ span $\fz$, there are matrices $A_{i0},\dots,A_{it}$ over $\kk$ (for $i=1,\dots,m$) such that
\begin{equation}\label{e:newfXi}
\Big(\pi''(x_i\otimes w_j\otimes w_k^\star)\Big)_{j,k=1}^d = A_{i0}+\fz_1A_{i1}+\cdots+\fz_tA_{it}.
\end{equation}
Since the first diagram in \eqref{e:diagram} commutes, the matrices $X_i$ representing $\rho(x_i)$ relative to the basis $w_1,\dots,w_d$ satisfy
$$X_i = A_{i0}+z_1A_{i1}+\cdots+z_tA_{it}.$$
Thus, we view \eqref{e:newfXi} as affine matrices over $\px{\fz_1,\dots,\fz_t}=\px{\fZ}$ and denote them $\fX_i$ as in \eqref{e:fXi}. 
By Proposition \ref{p:freetest} and \eqref{e:freeqsM} (see also \eqref{e:freeqs0g}), the fundamental relations between $z_1,\dots,z_t$ are 
\begin{equation}\label{e:freeqs}
(x_i^g w_j^g)(\fX)e_1 - \sum_{k=1}^d (w_k^g)(\fX)e_1 \cdot (\fX_i)_{kj} 
\quad \text{for all }i,j,g.
\end{equation}
Expressions \eqref{e:freeqs} are vectors of noncommutative polynomials in $\fz_1,\dots,\fz_t$; we claim that they are zero (i.e., they are trivial relations between $z_1,\dots,z_t$). 
By the definition of $\fX_i$ we have
$(\fX_i)_{kj} = \pi''(v\otimes w_j\otimes w_k^\star)$. Since the diagrams \eqref{e:diagram} commute, $(w_k^g)(\fX)e_1$ and $(x_i^g w_j^g)(\fX)e_1$ are precisely the vectors representing $\Pi(w_k^g)(1\otimes 1)$ and $\Pi(x_i^gw_j^g)(1\otimes 1)$, respectively, relative to the basis $w_1,\dots,w_d$ of $W$. Thus, the expressions \eqref{e:freeqs} are vector representations of the left-hand-side expressions in \eqref{e:freeqs1}. 
The equations \eqref{e:freeqs1} then imply that \eqref{e:freeqs} are zero. 
Hence, the fundamental relations between $z_1,\dots,z_t$ are trivial, so $z_1,\dots,z_t$ are free in $M$ by Proposition \ref{p:freetest}. In particular, they are also linearly independent, so 
$$\dim Z= \dim S= \dim (W\cap VW)^\perp= md-(d-1).$$
Hence, $t=md-(d-1)$ and $z_1,\dots,z_t$ are free generators of $L$.
\end{proof}

Theorem \ref{t:linfree} bears resemblance to the renowned Nielsen–Schreier index formula for subgroups of free groups. While the proof of Theorem \ref{t:linfree} does not rely on free groups (as opposed to its earlier special case \cite[Theorem 5.1]{KPPV}), an expert may nevertheless recognize analogies in the proof. 
For another manifestation of the index formula, see the Schreier-Lewin formula for ranks of free modules over free algebras \cite[Theorem 4]{Lew0}.

It is not surprising that invariants of infinite linear groups may not be finitely generated. The techniques developed in this paper do not readily apply to this case (which is left to later studies). Nevertheless, let us show that the freeness conclusion still holds in the simplest nontrivial example.

\begin{example}\label{ex:inf}
Suppose $\kk$ is infinite, and let $G$ be an infinite multiplicative subgroup of $\kk\setminus\{0\}$. Consider the linear action of $G$ on $\rx{x,y}$ given by $x^\alpha=\alpha x$ and $y^\alpha=y$ for $\alpha\in G$. We claim that $\rx{x,y}^G$ is a free skew field on infinitely many generators. Let $L$ be the skew subfield of $\rx{x,y}$ generated by $\{x^{-n}yx^n\colon n\in\Z \}$, and consider the automorphism $\sigma:L\to L$ given by $\sigma(\ell)=x^{-1}\ell x$. By \cite[Lemma 5.5.6]{Coh1}, the generators $x^{-n}yx^n$ for $n\in\Z$ are free in $L$, and $\rx{x,y}=L(x;\sigma)$ (the classical ring of right fractions of the skew polynomial ring $L[x;\sigma]$). We claim that $L=\rx{x,y}^G$. Clearly, $L\subseteq \rx{x,y}^G$. Conversely, if $r\in \rx{x,y}^G\subseteq L(x;\sigma)$, we can write $r=pq^{-1}$ for $p,q\in L[x;\sigma]$ of minimal total degree in $x$. Since $L[x;\sigma]$ is a principal right ideal domain \cite[Proposition 2.1.1]{Coh1}, $p$ and $q$ are uniquely determined up to a right multiple from $L$. Since
$$p(x)q(x)^{-1}=r=r^\alpha = p(\alpha x)q(\alpha x)^{-1}$$
for infinitely many $\alpha \in G$, uniqueness implies $p,q\in L$, and therefore $r\in L$.
\end{example}

\section{Non-free invariant skew subfields of a free skew field}\label{sec:nonfree}

The aim of this section is threefold. First, it demonstrates that there exist non-free invariant skew subfields in the free skew field, resolving a L\"uroth-like question of Schofield \cite{Sch}. Second, the number of required generators does not depend just on the group, but on its action (cf. Corollary \ref{c:nonlin} and Theorem \ref{theo:LinFinitelyGenerated}). Third, we disprove the conjecture of Cohn \cite{Coh0} on centralizers in free skew fields.

\subsection{Free de Jonqui\`eres involutions}\label{sec:jonq}

To establish the above claims, we advance the calculations from Section \ref{ex:Z2} for a special kind of involution on $\rx{x,y}$. Let $f\in\kk(x)\setminus\{0\}$.
Then,
$$x^\phi = y^{-1}xy,\quad y^\phi = y^{-1}f$$
determines an involution $\phi$ of $\rx{x,y}$, since
$$(x^\phi)^\phi=(y^\phi)^{-1}x^\phi y^\phi=f^{-1}y\cdot y^{-1}xy\cdot y^{-1}f=x,\quad
(y^\phi)^\phi=(y^\phi)^{-1}f(x^\phi)=f^{-1}y\cdot y^{-1}fy=y.$$
This involution is a noncommutative analog of a de Jonqui\`eres involution of the plane \cite{BB}. 
If $\kk$ is algebraically closed, one can restrict to $f\in\kk[x]$ with simple roots after a change of variables $y\mapsto qy$ for a suitable $q\in\kk(x)$. 
If $f$ is a polynomial of degree at least three with pairwise distinct roots, the commutative specialization of $\phi$ is not conjugated to a linear automorphism (see~\cite{BB}); thus, this involution is a natural candidate for seeking non-free skew subfields of $\rx{x,y}$. We assume $\kar\kk\neq 2$ for the rest of this section.

As in Section \ref{ex:Z2} we take $W=\operatorname{span}_{\kk}\{1,x\}$, and write
$$\sD_1=\begin{pmatrix}x&0\\0&y^{-1}xy\end{pmatrix},\quad
\sD_2=\begin{pmatrix}y&0\\0&y^{-1}f\end{pmatrix},\quad
\sW=\begin{pmatrix}1& x\\ 1& y^{-1}xy\end{pmatrix}.$$
Denote
$$
X:=\begin{pmatrix}0&z_1 \\ 1& z_2\end{pmatrix}=\sW^{-1}\sD_1\sW,\quad
Y:=\begin{pmatrix}z_3&z_5 \\ z_4& z_6\end{pmatrix}=\sW^{-1}\sD_2\sW,\quad
E=I\otimes e_1=\begin{pmatrix}1&0\\0&0\\0&1\\0&0\end{pmatrix}.
$$
Then, $z_1,\dots,z_6$ generate $\rx{x,y}^{\Z_2}$, and the fundamental relations between these generators come from \eqref{e:freeqsM} by Proposition \ref{p:freetest}. 
One can check that all six generators are nonzero by inspecting the formula \eqref{e:z2general} (we point this out because the calculations below involve multiplying equations by the generators or their inverses). The fundamental relations between $z_1,\dots,z_6$ as in \eqref{e:freeqsM} specialize to
\begin{align*}
\begin{pmatrix}X&0\\ 0&Y^{-1}XY\end{pmatrix}
\begin{pmatrix}I&X\\ I&Y^{-1}XY\end{pmatrix}
E
&=\begin{pmatrix}I&X\\ I&Y^{-1}XY\end{pmatrix}EX,\\
\begin{pmatrix}Y&0\\ 0&Y^{-1}f(X)\end{pmatrix}
\begin{pmatrix}I&X\\ I&Y^{-1}XY\end{pmatrix}
E
&=\begin{pmatrix}I&X\\ I&Y^{-1}XY\end{pmatrix}EY,
\end{align*}
which simplify to
\begin{equation}\label{e:rel1}
\begin{split}
\begin{pmatrix}X&X^2\\ XY&X^2Y\end{pmatrix}E
&=\begin{pmatrix}I&X\\ Y&XY\end{pmatrix}EX,\\
\begin{pmatrix}Y&YX\\ f(X)&f(X)Y^{-1}XY\end{pmatrix}E
&=\begin{pmatrix}I&X\\ Y&XY\end{pmatrix}EY.
\end{split}
\end{equation}
A direct calculation shows that the first block rows of both equations in \eqref{e:rel1} trivially hold; that is, they do not store any information about the relations between $z_1,\dots,z_6$. Thus, \eqref{e:rel1} simplifies to
\begin{equation}\label{e:rel2}
XC=CX,\quad f(X)Y^{-1}C=CY,\quad\text{where}\ C:=\begin{pmatrix}Y&XY\end{pmatrix}E
=\begin{pmatrix}z_3 & z_1z_4\\ z_4& z_3+z_2z_4 \end{pmatrix}.
\end{equation}
Since $Y$ and $C$ have the same first column,
we can write 
$$Y^{-1}C=\begin{pmatrix}1&a\\0&b\end{pmatrix}.$$
 Because $x$ and $y$ do not commute, $\sD_1$ and $\sD_2$ do not commute, and therefore $X$ and $Y$ do not commute.
We have $Y^{-1}C\neq I$, because otherwise the first equation in \eqref{e:rel2}
would imply that $X$ commutes with $Y=C$. 
Since $C$ commutes with $X$ and $(Y^{-1}C)^2=f(X)^{-1}C^2$ by \eqref{e:rel2}, it follows that $(Y^{-1}C)^2X=X(Y^{-1}C)^2$. 
The first columns of this matrix equation imply $a+ab=0$ and $b^2=1$. Hence, $(Y^{-1}C)^2=I$, and so $C^2=f(X)$ by \eqref{e:rel2}. 
Furthermore, we have $b=1$ and $a=0$ (because $\kar\kk\neq2$) or $b=-1$; the first option contradicts $Y^{-1}C\neq I$, and so $b=-1$. 

Let $f(X)=\left(\begin{smallmatrix} f_{11} & f_{12}\\ f_{21}& f_{22}\end{smallmatrix}\right)$. 
The first columns in $XC$ and $CX$ are trivially equal, and do not contribute any information. The equations $b=-1$ (i.e., the last piece of information from $(Y^{-1}C)^2=I$), $XC=CX$ and $C^2=f(X)$ are then equivalent to
\begin{align}
\label{e:rel_corner} \big(z_6-z_4z_3^{-1}z_5\big)^{-1}(z_3+z_2z_4-z_4z_3^{-1}z_1z_4)&=-1,\\
\label{e:rel_comm21} z_1z_3+z_1z_2z_4 &= z_3z_1+z_1z_4z_2,\\
\label{e:rel_comm22} z_1z_4+z_2z_3+z_2^2z_4&=z_4z_1+z_3z_2+z_2z_4z_2,\\
\label{e:rel_Csq11} z_3^2+z_1z_4^2&=f_{11},\\
\label{e:rel_Csq21} z_4z_3+z_3z_4+z_2z_4^2&=f_{21},\\
\label{e:rel_Csq12} z_3z_1z_4+z_1z_4z_3+z_1z_4z_2z_4&=f_{12},\\
\label{e:rel_Csq22} z_4z_1z_4+(z_3+z_2z_4)^2&=f_{22}.
\end{align}
Since $f(X)$ commutes with $X$, we have $f_{12}=z_1f_{21}$ and $f_{22}-f_{11}=z_2f_{21}$. 
Then, \eqref{e:rel_Csq12} is a consequence of $f_{12}=z_1f_{21}$, \eqref{e:rel_Csq21} and \eqref{e:rel_comm21} (right-multiplied by $z_4$); 
\eqref{e:rel_Csq22} is a consequence of $f_{22}-f_{11}=z_2f_{21}$, \eqref{e:rel_Csq11}, \eqref{e:rel_Csq21} and \eqref{e:rel_comm22} (right-multiplied by $z_4$).  
Observe that $z_6$ is a noncommutative rational function in $z_1,\dots,z_5$ by \eqref{e:rel_corner}, and that $z_5$ does not appear in any other relation. Thus, $\rx{x,y}^{\Z_2}$ is generated by $z_1,\dots,z_5$ where $z_1,\dots,z_4$ satisfy relations \eqref{e:rel_comm21}, \eqref{e:rel_comm22}, \eqref{e:rel_Csq11} and \eqref{e:rel_Csq21}.

From here on, we restrict to a cubic polynomial $f=x^3+\alpha_2 x^2+\alpha_1 x+\alpha_0$ with $\alpha_i\in\kk$. 
In this case, $f_{11}=z_1z_2+\alpha_2 z_1+\alpha_0$ and $f_{21}=z_1+z_2^2+\alpha_2z_2+\alpha_1$. 
We can express $z_1=z_4z_3+z_3z_4+z_2z_4^2-z_2^2-\alpha_2 z_2-\alpha_1$ from \eqref{e:rel_Csq21} -- this is the main reason for restricting to a cubic $f$. 
After inserting $z_1$ in \eqref{e:rel_comm22} and rearranging, we obtain
$$
\big(z_3+z_4(z_4^2-z_2-\alpha_2)\big)\cdot (z_4^2-z_2) 
= (z_4^2-z_2)\cdot \big(z_3+z_4(z_4^2-z_2-\alpha_2)\big).
$$
Denote $r=z_4^2-z_2-\alpha_2$ and $s=z_3+z_4(z_4^2-z_2-\alpha_2)$. 
Then, $r$ and $s$ commute, and $r,s,z_4$ generate the same subfield of $\rx{x,y}$ as $z_2,z_3,z_4$:
$$z_2=z_4^2-\alpha_2-r,\quad z_3=s-z_4r,\quad z_1=z_4s+sz_4-z_4rz_4-(r^2+\alpha_2r +\alpha_1).$$
Next, we express the remaining two equations \eqref{e:rel_Csq11} and \eqref{e:rel_comm21} in terms of $r,s,z_4$.
After a tedious yet straightforward calculation (using $rs=sr$), we simplify them to
\begin{align}
\label{e:last1} s^2&=r^3+\alpha_2 r^2+\alpha_1r+\alpha_0,\\
\label{e:last2} z_4(s^2-r^3-\alpha_2 r^2-\alpha_1r)&=(s^2-r^3-\alpha_2 r^2-\alpha_1r)z_4.
\end{align}
Note that \eqref{e:last2} is a direct consequence of \eqref{e:last1}. 

Let us summarize our findings: $z_6,z_1$ are noncommutative rational functions in $z_2,\dots,z_5$, and $z_2,z_3$ can be further replaced by $r,s$. 
Thus, $L$ is generated by $z_4,z_5,r,s$ and in terms of these generators, 
the fundamental relations reduce to $rs=sr$ and $s^2=f(r)$. 
Using the formulae \eqref{e:z2general}, it is easy to verify that $r$ and $s$ are nonconstant (and thus transcendental) elements of $\rx{x,y}$ (for their simplified explicit forms, see Corollary \ref{c:centraliser} below). 
Hence, they generate the function field of the plane cubic curve $\cC=\{(u,v):v^2=f(u)\}$.

Given a $\kk$-algebra $A$, its universal derivation bimodule $\Omega_\kk(A)$ is the $A$-subbimodule in $A\otimes_{\kk} A^\op$ generated by $\{\d a\colon a\in A\}$, where $\d a =a\otimes 1-1\otimes a$. If $A$ is a skew field generated by $a_1,\dots,a_m$, then $\Omega_\kk(A)$ is generated by $\d a_1,\dots,\d a_m$ since $\d:A\to\Omega_\kk(A)$ is a derivation. 
For the free skew field $M=\kk\lfree x_1,x_2,\dots,x_m\rfree$, the bimodule
$\Omega_\kk(M)$ is freely generated by $\d x_1,\d x_2,\dots,\d x_m$ \cite[Proposition 5.8.7 and Theorem 5.8.10]{Coh1}. The  universal derivation bimodule can be used to test freeness.

Assume that $f$ has simple roots and $\kar \kk\neq 2$, so that this curve $\cC=\{(u,v): v^2=f(u)\}$ is smooth (in other words, its completion is an elliptic curve).
Then the function field $F$ of $\cC$ is not rational over $\kk$ \cite[Example II.8.20.3]{Har} (which implies that it is not a free skew field). We will need the following lemma later (even though we already know that $F$ is not a free skew field):
\begin{lemma}\label{l:notcyclic}
   The $F$-bimodule $\Omega_{\kk}(F)$ is not cyclic.
\end{lemma}
\begin{proof}
 The projective closure $\overline\cC$ of $\cC$ with the point at infinity $\infty$ is an elliptic curve. Let $R=\kk[u,v]/(v^2-f(u))$ be the coordinate ring of $\cC$ (with $F$ as its field of fractions). Suppose $\Omega_\kk(F)$ is a cyclic $F$-bimodule, in other words, a principal ideal in $F\otimes_{\kk} F$. Then, it is generated by some $c\in F\otimes_\kk R$. Note that $\Omega_\kk(F)$ is the localization (at $1\otimes R\setminus\{0\}$) of the maximal ideal $(u-r,v-s)$ in $F\otimes_\kk R$ corresponding to the $F$-rational point $P=(r,s)$ on $\cC$. The principal divisor of $c$ is then of the form $\divisor(c)=(P)+(Q_1)+\cdots+(Q_m)-(m+1) (\infty)$, where $Q_1,\dots,Q_m$ are $\kk$-rational points on $\cC$. 
Because $Q\mapsto (Q)-(\infty)$ induces a one-to-one correspondence between the points on $\overline\cC$ and the equivalence classes of degree-zero divisors on $\overline\cC$ \cite[Example IV.1.3.7]{Har}, there is a $\kk$-rational point $Q$ on $\overline\cC$ and $d\in F\setminus\{0\}$ (viewed as the function field of $\cC$) such that $m(\infty)-(Q_1)-\cdots-(Q_m)=(Q)-(\infty)+\divisor(d)$. Then, $\divisor(cd)=(P)-(Q)$; note that $P\neq Q$ since $Q$ is $\kk$-rational and $P$ is not. 
But distinct points on a non-rational smooth complete curve are not linearly equivalent \cite[Example II.6.10.1]{Har}, so $(P)-(Q)$ cannot be a principal divisor, a contradiction. Thus, $\Omega_\kk(F)$ is not a cyclic $F$-bimodule.
\end{proof}

The following Theorem gives us an example of an invariant subfield of a free skew field that is {\em not} free.

\begin{theorem}\label{t:nonfree}
Assume $\kar\kk\neq2$ and let $f\in\kk[x]$ be a monic cubic polynomial with simple roots. 
If $\Z_2$ acts on $\rx{x,y}$ via the involution
$$x\mapsto y^{-1}xy,\qquad y\mapsto y^{-1}f,$$
then the skew field $\rx{x,y}^{\Z_2}$ is generated by four (but not three) elements, and is not free. 
\end{theorem}

\begin{proof}
Let $K=\rx{z_4,z_5}$, and let $F$ be field generated by $r,s$ satisfying $s^2=f(r)$.  
Let $D$ be the skew field coproduct of $K$ and $F$ over $\kk$; that is, $D$ is the universal skew field of fractions of the ring coproduct of $K$ and $F$ over $\kk$ (since the ring coproduct of $K$ and $F$ is a free ideal ring \cite[Theorem 5.3.9]{Coh1}).
Thus, $D$ is generated by $z_4,z_5,r,s$. The preceding calculations in this section show that $L$ is also generated by $z_4,z_5,r,s$, and all the relations between these generators in $D$ also hold in $L$ (and a priori, they may satisfy further relations in $L$). Thus, there is a local homomorphism from $D$ to $L$. However, all the preceding calculations are reversible (i.e., no information was lost in manipulating the equations), so $D$ satisfies the fundamental relations of $L$ (when we choose the generators of $D$ in the same way as $r,s$ express with $z_1,z_2,z_3,z_4$ in $L$). Thus, $D\cong L$ by Proposition \ref{p:freetest}.

The $K$-bimodule $\Omega_\kk(K)$ is free of rank 2 with generators $\d z_4,\d z_5$ by \cite[Proposition 5.8.7 and Theorem 5.8.10]{Coh1}. 
The $F$-bimodule $\Omega_\kk(F)$ is generated by $\d r,\d s$ but it is not free of rank 1 
by Lemma~\ref{l:notcyclic}. 
By \cite[Theorems 5.8.8 and 5.8.10]{Coh1}, the $D$-bimodule $\Omega_\kk(D)$ is generated by $\d z_4,\d z_5,\d r,\d s$ and the relations between these are precisely those from $\Omega_\kk(K)$ and $\Omega_\kk(F)$. 
Thus, $\Omega_\kk(D)$ is not a free $D$-bimodule and cannot be generated by three elements. Thus, the skew field $D$ cannot be generated by three elements, and is not free by \cite[Theorem 5.8.10]{Coh1}.
\end{proof}

\subsection{Centralizers in free skew fields}\label{sec:cent}

The invariant skew field in Section \ref{sec:jonq} is also significant for centralizers in free skew fields.
The centralizer of a non-scalar element in a free skew field is commutative \cite[Corollaire 2]{Coh0}, and furthermore finitely generated of transcendence degree 1 over $\kk$ \cite[Theorem 11.6]{Sch}. That is, every centralizer of a non-scalar element is isomorphic to the function field of a curve. 
It was conjectured in \cite[Conjecture]{Coh0} that the centralizer is rational, i.e., isomorphic to $\kk(t)$. 
However, this is not the case, as demonstrated by the $r,s$ constructed above that generate a non-rational field (and consequently their centralizer is not rational, since all subfields of $\kk(t)$ containing $\kk$ are rational by L\"uroth's theorem \cite[Example 2.5.5]{Har}).

\begin{corollary}\label{c:centraliser}
Assume $\kar\kk\neq2$ and let $f\in\kk[x]$ be a monic cubic polynomial with simple roots and the quadratic coefficient $\alpha$. 
Define $r,s\in\rx{x,y}\setminus\kk$ as
\begin{align*}
&r=b_1^2-b_2-\alpha\quad\text{and}\quad
s=y-xb_1+b_1r,\\
&\text{where}\qquad b_1=(xy-yx)^{-1}(f-y^2),\ b_2= (xy-yx)^{-1}(x^2y-yx^2).
\end{align*}
Then, $rs=sr$ and $s^2=f(r)$; the centralizer of $r$ (and $s$) in $\rx{x,y}$ is not rational.
\end{corollary}

For a comparison, the centralizer of a non-scalar element in a free algebra is isomorphic to $\kk[t]$ by Bergman's centralizer theorem \cite[Theorem 5.3]{Ber}; moreover, its centralizer within the corresponding free skew field is isomorphic to $\kk(t)$ by \cite[Theorem 7.9.8]{Coh2}. 
One can push this analogy further: \cite[Corollary 6.7.2]{Coh2} shows that the centralizer of a non-scalar element in the ring of formal power series in free variables is isomorphic to $\kk[[t]]$; it is also not hard to see that the centralizer of a non-scalar element in the skew field of Malcev-Neumann series over a free group is isomorphic to the field of formal Laurent series $\kk(\!(t)\!)$.

Corollary \ref{c:centraliser} shows that all function fields of curves of genus 1 (at least when $\kk$ is algebraically closed and $\kar\kk\neq2$) embed into a free skew field; it is not known what happens for curves with higher genera. 
A remarkable result \cite[Theorem 11.10]{Sch} states that two elements of a free skew field are either free or commute; this fact hints at a possibly special role of commutativity among non-trivial relations in a free skew field. Thus, we conclude the section with the following questions.

\begin{question}
Let $M$ be a noncommutative free skew field.
\begin{enumerate}[(a)]
\item Does the function field of every curve embed into $M$?
\item Is every skew subfield of $M$ a skew field coproduct (over $\kk$) of function fields of curves?
\end{enumerate}
\end{question}

A better understanding of invariants for the involution $x\mapsto y^{-1}xy,y\mapsto y^{-1}f$ for a general $f\in\kk[x]$ would provide further insight into these questions. Yet at the moment, a conclusive procedure for reducing equations in free skew fields is lacking.

There is also a geometric explanation for why elements in the free skew field with centralizers of higher genera are rather atypical. Namely, a higher genus of the centralizer of $r\in\rx{x_1,\dots,x_m}$ implies a higher algebraic dependency of eigenvalues of $r$ (as a multivariate matrix function), as follows.

\newcommand{\uX}{X}

\begin{proposition}\label{p:genus}
Let $\kk$ be an algebraically closed field of characteristic~0. Suppose that $r,s\in\rx{x_1,\dots,x_m}$ commute and generate the function field of a curve of genus $g$. If $n\ge g$ is such that $r$ and $s$ are defined at some point in $(\kk^{n\times n})^m$, then the eigenvalues of $r$ on $(\kk^{n\times n})^m$  
generate a field of transcendence degree at most $n-g$ over $\kk$.
\end{proposition}

\begin{proof}
    The commuting variety $\cZ$
    is the set of all pairs $(A,B)\in (\kk^{n\times n})^2$ with $AB=BA$.
    It is an irreducible variety of dimension $n^2+n$.\footnote{Since we view $\cZ$ as an algebraic set it is also reduced. It is not known whether the commuting variety, as a scheme, is reduced.} 
The group $\GL_n(\kk)$ acts on $\cZ$ by simultaneous conjugation, and $S_n$ acts on $\cD\times \cD$ by simultaneously permuting the diagonal entries of the two matrices.
There is an isomorphism between GIT quotients
$\cZ/\!\!/\GL_n(\kk)\cong (\cD\times \cD)/\!\!/S_n$ by \cite[Lemma 2.8.3]{GG}.
We have a rational map $(r,s):(\kk^{n\times n})^m\dashto \cZ$. Define the rational map $\sigma:(\kk^{n\times n})^m\dashto (\cD\times \cD)/\!\!/S_n$ as the composition
$$
\xymatrix{
(\kk^{n\times n})^m\ar@{-->}[r] & \cZ\ar[r]  &\cZ/\!\!/\GL_n(\kk)\cong(\cD\times \cD)/S_n}
$$    
Since $r(\uX)$ and $s(\uX)$ commute for every tuple of $n\times n$ matrices $\uX$ where they are both defined, each pair of their corresponding eigenvalues give a point on $\cC$. So the image of $\sigma$ is contained in $\cC^{(d)}=\cC^d/\!\!/S_d$, the $d$\textsuperscript{th} symmetric power of $\cC$ (see~\cite[Proposition 3.1]{Mil2}).
 There is a canonical map $\phi:\cC^{(n)}\to J(\cC)$ by \cite[Section 5]{Mil2}, where $J(\cC)$ is the Jacobian variety of $\cC$, and $\phi$ is surjective since $n\ge g$. Note that $\dim \cC^{(n)}=n$ and $\dim J(\cC)=g$ by \cite[Proposition 2.1]{Mil2}.
Since $J(\cC)$ is an abelian variety and every rational map from an affine space to an abelian variety is constant by \cite[Corollary 3.9]{Mil1}, it follows  that $\phi\circ\sigma$ is constant. Therefore, 
the image of $\sigma$ has dimension at most $n-g$. The projection of $\cD\times \cD\to \cD$ onto the first factor induces a morphism of quotients $\pi:(\cD\times \cD)/\!\!/S_n\to \cD/\!\!/S_n$. The composition $\pi\circ \sigma$ sends an $m$-tuple $\uX\in (\kk^{n\times n})^m$ to the set of eigenvalues of $r(X)$. The dimension of the image of $\pi\circ\sigma$ is at most the dimension of the image of $\sigma$, which is $n-g$.
\end{proof}

Proposition \ref{p:genus} does not admit a converse; for instance, eigenvalues of $xy-yx$ add up to 0 and are thus linearly dependent, and the centralizer of $xy-yx$ in $\rx{x,y}$ equals $\kk(xy-yx)$ by \cite[Theorem 7.9.8]{Coh2}. The authors do not know of an element in a free skew field whose eigenvalues (in the sense of matrix functions as above) satisfy more than a single relation.

\section{Algorithms}\label{sec:algo}

In this section we discuss computational aspects of finding generators for skew subfields of $M=\rx{x_1,\dots,x_m}$. 
Denote $R=\px{x_1,\dots,x_m}$.

First, we outline some basics of computing and rational identity testing in $M$. 
For every $r\in M$ there exist $\delta\in\N$, 
$\delta\times \delta$ matrices $A_0,A_1,\dots,A_m\in \kk^{\delta\times\delta}$
and row/column vectors $u,v$ over $\kk$ such that 
$r=uA^{-1}v$ and
$A=A_0+x_1A_1+\cdots+x_mA_m$ is full.
The triple $(u,A,v)$ is known as a \emph{linear representation} of $r$ of dimension $\delta$. See \cite[Section 1]{CR} for the construction of linear representations, and the arithmetic operations in $M$ in terms of linear representations. Then, $r=0$ in $M$ if and only if the $(\delta+1)\times (\delta+1)$ matrix 
$$
\begin{pmatrix}
0 & u\\
v & A\end{pmatrix}
$$
is not full over $R$. There is a polynomial-time deterministic algorithm for deciding non-fullness of such a matrix by \cite[Theorem I.1]{Gar} and \cite[Theorem 1.5]{Iva}. If $\kk$ is infinite, then a $(\delta+1)\times (\delta+1)$ affine matrix over $R$ is full if and only if its evaluation at some tuple of matrices $X\in (\kk^{\delta\times \delta})^m$ is invertible \cite[Proposition 1.12]{DM}. In particular, this allows for a straightforward probabilistic test of $r\neq0$: based on the number of operations in $r$ one can estimate the dimension of a linear representation of $r$, and check if $r(X)$ is nonzero at a sufficiently large random matrix tuple $X$ over $\kk$. 

The above algorithms for rational identity testing also extend to computing the rank of matrices over $M$. Given a rectangular matrix $A$ over $R$, there is an algorithm for deciding whether $A$ has full rank over $M$: pad $A$ with zeros to a square matrix, linearize it, and then apply \cite[Theorem 1.5]{Iva}. Furthermore, if $\kk$ is infinite, then $A$ has full rank over $M$ if and only if $A(X)$ has full rank for some $X\in (\kk^{N\times N})^m$, where $N\in\N$ is sufficiently large (and can be estimated using the number of operations in $A$ and \cite[Proposition 1.12]{DM}).

Let $L$ be a skew subfield of $M$ for which $d=\dim M_L$ is finite. Below we present an algorithm that finds a set of generators with at most $d^2(m-1)+d$ elements. 
The algorithm finds a basis $\cB$ of $M$ as a right $L$-module with $1\in \cB$. Such a basis has $d=\dim M_L$ elements.
If $\cB=\{b_1,\dots,b_d\}$ with $b_1=1$, then the algorithm finds $d \times d$ matrices $X_1,\dots,X_m$ with entries $(X_\ell)_{i,j}\in L$ that satisfy $x_\ell b_j=\sum_{i} b_i (X_\ell)_{i,j}$. 
Because of the way the basis $\cB$ is constructed in the algorithm, some of the coefficients $(X_\ell)_{i,j}$ will be $0$ or $1$. 
Let $\cL$ be the set of all elements $(X_\ell)_{i,j}$ distinct from $0$ and $1$; then, $\cL$ generates the skew field $L$.

The input of Algorithm \ref{algo1} are the field $\kk$, the number $m\in\N$ determining $M$, and the intermediate skew field $L$. 
The latter is given indirectly; we assume there is an algorithm for $w,b_1,\dots,b_r\in M$ that finds coefficients $c_1,\dots,c_r\in L$
with $w=b_1c_1+\cdots+b_rc_r$ or decides that such coefficients do not exist. Note that the coefficients are unique if $b_1,\dots,b_r$ are independent as right $L$-module generators.

During the execution of Algorithm \ref{algo1}, $\cB$ is always a set of elements of $M$ that are linearly independent over $L$.
This set consists of words in $x_1,\dots,x_m$. 
If $\cB=\{b_1,\dots,b_r\}$ with $b_1=1$, then for every $i>1$, $b_i$ equals $x_\ell b_j$ for some $\ell$ and some $j<i$; therefore, $(X_\ell)_{i,j}=1$
and $(X_\ell)_{t,j}=0$ for $t\neq i$. At termination, $\cB$ is a basis of $M$ as a right $L$-module. At least $d(d-1)$
entries $(X_\ell)_{i,j}$ are equal to $0$ or $1$. This implies that $\cL$ has at most $d^2m-d(d-1)=d^2(m-1)+d$ elements at termination. During execution, $\cT$ consists of monomials that have to be checked whether they lie in the $L$-module generated by $\cB$.
\begin{algorithm}\label{algo1}
\phantom{}
\begin{algorithmic}[1]
\State $\cT=\{1\}$
\State $\cL=\{\}$
\State $\cB=\{\}$
\While{$\cT\neq \{\}$}
    \State    choose $w\in \cT$
    \State $\cT=\cT\setminus\{w\}$
    \If{$w\in \cB L$}
       \State  write $w=\sum_{b\in \cB} b c_b$ with $c_b\in L$ for all $b\in \cB$
        \State $\cL=\cL\cup \{c_b\colon b\in \cB\}$
    \Else
        \State $\cB=\cB\cup \{w\}$
        \State $\cT=\cT\,\cup\,\{x_1w,x_2w,\dots,x_mw\}$
    \EndIf
\EndWhile
\end{algorithmic}
\end{algorithm}

The correctness and termination of Algorithm \ref{algo1} follows from the proofs of Proposition \ref{prop:Ws} and Theorem \ref{theo:FinitelyGenerated}. Let us concretize lines 7--8 in the case where $L=M^G$ for a finite group $G$. Given $b_1,\dots,b_\ell\in M$ with $\ell\le |G|$, let us call the $|G|\times\ell$ matrix over $M$
$$\big(b_j^g\big)_{g\in G,1\le j\le \ell}$$
the $G$-matrix of $b_1,\dots,b_\ell$ (determined up to ordering of $G$). Similarly to Lemma \ref{l:Winvert}, we see that $b_1,\dots,b_\ell \in M$ are linearly independent over $L$ if and only if the $G$-matrix of $b_1,\dots,b_\ell$ has full rank over $M$. 
Then, the equation $w=\sum_{b\in \cB} b c_b$ for known $w\in\cT$, $b\in\cB$ and unknown $c_b\in L$ in line 8 is equivalent to
\begin{equation}\label{e:overdetermined}
(w^g)_{g\in G} = \big(b_\ell^g\big)_{g\in G,b\in\cB} \cdot (c_b)_{b\in \cB}.
\end{equation}
Note that \eqref{e:overdetermined} is solvable for $c_b\in L$ if and only if it is solvable for $c_b\in M$, since the $G$-matrix of $\cB$ has full rank. 
Thus, \eqref{e:overdetermined} is an overdetermined right-linear system of equations over $M$, and can be solved via Gaussian elimination (using rational operations and identity testing in $M$).

Now let $G$ be a finite group acting faithfully and linearly on $V=\operatorname{span}_{\kk}\{x_1,\dots,x_m\}$, and $\kar\kk\nmid |G|$. The skew field $L=M^G$ is free of rank $|G|(m-1)+1$ by Theorem \ref{t:linfree}. In this setup, the key advantage in comparison with Theorem \ref{theo:FinitelyGenerated} is the existence of a $G$-stable subspace $W\subseteq R$ such that $M=WL$, $\dim W=|G|$ and 
$1\in W\subseteq\kk+VW$ (Proposition \ref{prop:Ws_lin} and Section \ref{ss:free}). 
Below we present an algorithm for finding a basis $\mathcal B$ of such $W$, consisting of homogeneous elements in $R$, and a set of free generators $\mathcal L$ of $L$. 
The input of Algorithm \ref{algo2} below are the field $\kk$, the number $m\in\N$ determining $M$, and the linear action of $G$ on $V$ (e.g., given as a collection of $|G|$ $m\times m$ matrices over $\kk$). 

\begin{algorithm}\label{algo2}
\phantom{}
\begin{algorithmic}[1]
\State $\cT=\{1\}$
\State $\cL=\{\}$
\State $\cB=\{\}$
\While{$\cT\neq\{\}$}
\State extract a basis of $\operatorname{span}_\kk\{x_iw\colon 1\le i\le m,\ w\in \cT\}$ as a disjoint union $\cB_1\cup\cdots \cup\cB_r$ such that $\operatorname{span}_\kk\cB_j$ is an irreducible $G$-space for $j=1\dots,r$
\State $\cT=\{\}$
    \For{$j=1,\dots,r$}
        \State choose one $w\in \cB_j$
        \If{$w\in \cB L$}
               \State  write $w=\sum_{b\in \cB} b c_b$ with $c_b\in L$ for all $b\in \cB$
                \State $\cL=\cL\cup \{c_b\colon b\in \cB\}$
        \Else
            \State $\cB=\cB\cup\cB_j$
            \State $\cT=\cT\cup \cB_j$
        \EndIf
    \EndFor
\EndWhile
\State replace $\cL$ by its maximal $\kk$-linearly independent subset
\end{algorithmic}
\end{algorithm}

The correctness and termination of Algorithm \ref{algo2} is established as in the case of Algorithm \ref{algo1}, with modifications according to the proof of Proposition \ref{prop:Ws_lin}. Let us highlight the main differences. 
First, $\operatorname{span}_\kk\{x_iw\colon 1\le i\le m,\ w\in \cT\}$ in line 5 is a $G$-space, since it is the tensor product of $G$-spaces $V$ and $\operatorname{span}_\kk\cT$ (and the latter is a $G$-space inductively). For finding its decomposition into irreducible $G$-spaces and extracting their bases, see \cite[Theorem 1.2]{BR} and \cite[Section 7.4]{HBO}. 
Second, lines 9--11 are applied to a single element of $w\in \cB_j$; this is sufficient according to the system \eqref{e:overdetermined} since the $\kk$-span of $\cB$ is $G$-stable.
Third, Theorem \ref{t:linfree} guarantees that line 18 results in a set of free generators of $L$; let us thus comment on testing $\kk$-linear dependence in $M$. 
By \cite[Theorem 7.6.16]{Row}, linear dependence of $r_1,\dots,r_\ell \in M$ is equivalent to vanishing of the Capelli polynomial in $r_1,\dots,r_\ell$ and $\ell-1$ auxiliary variables, and thus reduces to rational identity testing as in the beginning of the section. Alternatively, $r_1,\dots,r_\ell$ are linearly dependent if and only if their evaluations are linearly dependent on all sufficiently large matrix tuples \cite[Theorem 6.5]{Vol}.

\end{document}